\documentclass[12pt]{amsart}

\DeclareMathOperator{\cof}{\cof}

\usepackage[pagewise]{lineno}

\usepackage[english]{babel}
\usepackage[utf8]{inputenc}
\usepackage{amsmath}
\usepackage{amssymb}
\usepackage{amsfonts}
\usepackage{amsthm}
\usepackage{mathrsfs}
\usepackage[all]{xy}
\usepackage[pdftex]{graphicx}
\usepackage{color}
\usepackage[pdftex,dvipsnames]{xcolor} 
\usepackage{cite}
\usepackage{url}
\usepackage{indent first}
\usepackage[labelfont=bf,labelsep=period,justification=raggedright]{caption}
\usepackage[english]{babel}
\usepackage[utf8]{inputenc}
\usepackage[colorlinks=true, allcolors=blue]{hyperref}
\usepackage[colorinlistoftodos]{todonotes}
\usepackage{tkz-fct}
\usepackage{mathdots}
\usepackage{comment}
\usepackage{float}
\usepackage[T1]{fontenc}
\usepackage[utf8]{inputenc}

\usepackage{tgtermes}

\usepackage{algpseudocode}
\usepackage{mathrsfs}
\usepackage{bbm}
\usepackage{dsfont}
\usepackage{theoremref}
\usepackage[dvipsnames]{xcolor}

\usetikzlibrary{automata,positioning}

\usepackage{comment}

\newcommand{\bn}{\mathbb{N}}

\theoremstyle{plain}

\newtheorem{lemma}{Lemma}[section]

\newtheorem{example}{Example}[section]
\newtheorem{remark}{Remark}[section]

\newcommand{\Z}{\mathbb{Z}}

\newcommand{\N}{\mathbb{N}}

\theoremstyle{plain}
\newtheorem{thm}{Theorem}

\newtheorem{lem}[thm]{Lemma}
\newtheorem{cor}[thm]{Corollary}
\newtheorem{conj}[thm]{Conjecture}
\newtheorem{prop}[thm]{Proposition}

\newtheorem{qn}[thm]{Question}

\theoremstyle{definition}
\newtheorem{defn}[thm]{Definition}

\theoremstyle{remark}

\numberwithin{equation}{section}
\numberwithin{thm}{section}

\begin{document}

\title[On finite spacer rank for words and subshifts] {On finite spacer rank for words and subshifts}

\author[Gao]{Su Gao}
\address[Su Gao]{\\ School of Mathematical Sciences and LPMC, Nankai University, Tianjin 300071, P. R. China}
\email{sgao@nankai.edu.cn}

\author[Jacoby]{Liza Jacoby}
\address[Liza Jacoby]{ \\
     Department of Mathematics\\ UC Berkeley, Berkeley, CA 94720, USA}
\email{liza@math.berkeley.edu}

\author[Johnson]{William Johnson}
\address[William Johnson]{ \\
     Department of Mathematics \\ University of Maine, Orono, ME 04469, USA}
\email{william.l.johnson@maine.edu}

\author[Leng]{James Leng}
\address[James Leng]{\\ Department of Mathematics\\
    UCLA, Los Angeles, CA 90095, USA.}
\email{jamesleng@math.ucla.edu}

\author[Li]{Ruiwen Li}
\address[Ruiwen Li]{\\ School of Mathematical Sciences and LPMC, Nankai University, Tianjin 300071, P. R. China}
\email{rwli@mail.nankai.edu.cn}

\author[Silva]{Cesar E. Silva}
\address[Cesar E. Silva]{\\ Department of Mathematics\\
     Williams College \\ Williamstown, MA 01267, USA}
\email{csilva@williams.edu}

\author[Wu]{Yuxin Wu}
\address[Yuxin Wu ]{\\
     Management Science \& Engineering\\
Huang Engineering Center\\ Stanford University\\
475 Via Ortega\\
Stanford, CA 94305, USA}
\email{yuxinwu@stanford.edu}


\date{\today}

\begin{abstract}  
We define a notion of rank for words and subshifts that we call spacer rank, extending the notion of rank-one symbolic shifts of Gao and Hill. We construct infinite words of each finite spacer rank, of unbounded spacer rank, and show there exist words that do not have a spacer rank construction. 
We consider words that are fixed points of substitutions and give explicit conditions for the word to have an at most spacer rank two construction, and not to be rank one. We prove that finite spacer rank subshifts have topological entropy zero, and that there are zero entropy subshifts not defined by a word with a finite spacer rank construction.  We also study shift systems associated with infinite words, including those associated to Sturmian sequences, which we show are spacer rank-two systems. 
\end{abstract}

 \maketitle
	 
\section{Introduction} Rank-one measure-preserving transformations have played an important role in ergodic theory since the pioneering work of Chac\'on~\cite{Ch67,Ch69}. The terminology ``rank one'' comes from rank-one cutting and stacking systems \cite{Ch69, ORW82}. As shown by Kalikow \cite{Ka84}, one can encode cutting and stacking systems as a shift on a symbolic system; he also shows that the two systems are measurably isomorphic when the symbolic sequence is aperiodic. Symbolic models for measure-preserving transformations have also been introduced for higher finite rank cases, and have been used extensively in ergodic theory; Ferenczi \cite{Fe97} is a comprehensive survey of these results, and we also refer to King \cite{Ki88} and King--Thouvenot \cite{KiTh91}.

Motivated by these notions, researchers began to consider problems about symbolic systems as topological dynamical systems. In \cite{Bo13}, Bourgain studied a class of rank-one symbolic shifts. Rank-one symbolic shifts are also considered in \cite{AdFePe17,AbLeRu14,Da19}.
It was in \cite{GaHi16AT} that Gao and Hill started a systematic study of (non-finite) rank-one subshifts as topological dynamical 
systems, and proved several properties for them. In this paper we generalize this study to higher rank subshifts. We note that in \cite{DM08} Downanrowicz and Maass proposed a notion of topological rank for topological systems, that applies to subshifts but it is different from the notion we study as we mention below. 

We begin by considering one-sided infinite (binary) words and define a notion of rank for them that we call spacer rank, generalizing the definition of rank one from \cite{GaHi16AT}. In Section~\ref{S:wordrank} we prove some basic results about this notion of rank for infinite words. In fact, for each $n\geq 1$ we define a notion of spacer rank $n$ for subshifts, and a notion for infintie words that we refer to as having a \textit{spacer rank-$n$ construction}.
We also define a more restrictive notion called a \textit{proper spacer rank-$n$ construction} and elucidate what it means for a word to have the proper condition as opposed to the standard spacer rank-$n$ construction (for rank one the proper condition does not introduce any restrictions).
For each $n\geq 1$ we construct words with a proper spacer rank-$n$ construction. (By our definition, a word that has a (proper) spacer rank-$(n+1)$ construction does not have a (proper) spacer rank-$n$ construction.) We also define related notions of having a spacer rank construction and having unbounded spacer rank.
We construct words of unbounded spacer rank and show that there exist words, such as full complexity words, that do not admit any spacer rank construction.

In Section~\ref{S:sub} we consider a natural class of infinite words that turn out to have an at most spacer rank-two construction, namely substitution sequences with the alphabet $\{0, 1\}$. Generalized Morse sequences are examples of such words. 

Starting from Section~\ref{S:finiterank} we consider finite spacer rank subshifts as topological dynamical systems. More specifically, we consider subshifts defined by a one-sided infinite word and define the notion of spacer rank for such systems.
These subshifts are shift spaces in the sense of \cite{LiMa95}, but we show that the converse is not true, i.e., that there exist subshifts not arising from an infinite word.
Generalizing a well-known result for rank-one subshifts, we prove that all finite spacer rank subshifts have zero topological entropy. In contrast, there are words not of full complexity, in fact of polynomial complexity, that do not admit a spacer rank construction. Thus our notion of spacer rank provides a refined hierarchy for zero-entropy words and systems.

It is worth noting that for (symbolic) subshifts our notion of spacer rank is different from topological rank. For example, it is well-known that topological rank one maps are equicontinuous \cite[Theorem 6.3.6]{DuPe22}, while it was shown in \cite{GaZi18} that the maximal equicontinuous factor of a rank-one subshift is finite. (The notion of rank one and spacer rank one coincide.) As another example, the Morse system has topological rank three \cite[Example 6.3.8]{DuPe22}, while as shown later it has spacer rank two.
On the other hand, Sturmian systems have topological rank two \cite[Corollary 7.2.4]{DuPe22} and we show that they also have spacer rank two.

In Sections~\ref{S:sturmian} and \ref{S:charrank2} we consider more examples of spacer rank-two words and systems. In Section~\ref{S:sturmian} we prove that all Sturmian words have a proper spacer rank-two construction, and that subshifts generated by Sturmian words have spacer rank two (in particular they do not have (spacer) rank one). In Section~\ref{S:charrank2} we give additional examples of spacer rank-two systems and give a characterization of when a subshift generated by a spacer rank-two word has spacer rank two as a topological dynamical system. We end with a spacer rank-two system that has at least four orbit closures, in contrast to rank-one systems that have at most two orbit closures; hence this system cannot be topologically isomorphic to a rank-one system.

There are many questions regarding finite spacer rank words and subshifts that are left open by this paper. We hope that our results here will stimulate more research on this topic.

\smallskip

\textbf{Acknowledgments:} This work started as the 2019 undergraduate thesis of Y.W. under the direction of C.S., and contains part of the 21--22 undergraduate theses of L.J and W.J. and the 22--23 undergraduate thesis of R.L. In summer 2019, J.L. worked on this project as a research assistant at Williams College. S.G. acknowledges the National Natural Science Foundation of China (NSFC)
grants 12250710128 and 12271263 for partial support of his research.
From August 2019 to August 2021, C.S. served as a Program Director in the Division of Mathematical Sciences at the National Science Foundation (NSF), USA, and as a component of this job, he received support from NSF for research, which included work on this paper. Any opinion, findings, and conclusions or recommendations expressed in this material are those of the authors and do not necessarily reflect the views of the National Science Foundation. 

We thank Tom Garrity for asking about the rank of Sturmian sequences, and Sumun Iyer for useful suggestions. We would like to thank Houcein el Abdalaoui, Terry Adams, Sam Chistolini, Darren Creutz, Alexandre Danilenko, and Sebastien Ferenczi for comments. We also thank the anonymous referees for comments and suggestions.
 

\section{Spacer Rank Constructions for Words}\label{S:wordrank}

In this section we study spacer rank constructions of one-sided infinite words; we consider symbolic subshift systems generated from infinite words of any rank in Section~\ref{S:finiterank}. We let $\N$ denote the nonnegative integers, $\N^+$ the positive integers, and $\Z$ the set of integers.   

We start by extending to arbitrary finite rank the definition of rank-one words given by Gao and Hill in \cite{GaHi16B}. All words that we consider in this article are over the binary alphabet $\{0,1\}$. A \textbf{finite word} is an element of $\bigcup_{n=1}^\infty \{0,1\}^n$; if $w\in \{0,1\}^n$ we say $w$ has \textbf{length} $n$. A \textbf{word}, or \textbf{infinite word} is an element of $\{0,1\}^{\N}$, and a \textbf{bi-infinite word} is an element of $\{0,1\}^\Z$.
A \textbf{finite subword}, or a \textbf{factor}, of a word $V\in\{0,1\}^{\N}$ is a finite word of the form $V(i)V(i+1)\cdots V(i+k)$ for some $i,k\in\N$. If $u,v$ are finite words then $uv$ consists of the finite word $u$ followed by the finite word $v$ (i.e., $uv(i)=u(i)$ for $i\in\{0,\ldots, |u|-1\}$ and $uv(|u|+i)=v(i)$ for $i\in\{0,\ldots, |v|-1\}$). This notion is extended in a analogous way to $uV$ when $u$ is a finite word and $V$ is an infinite word. When clear from the context we may write \textit{word} instead of \textit{finite word}.

\begin{defn}\label{D:built} Let $\mathcal F$ denote the set of all finite words that start and end with $0$. Let $S$ be a finite subset of $\mathcal F$ and $w$ a finite word. A \textbf{building} of $w$ from $S$ consists of a sequence $(v_1,\ldots,v_k,v_{k+1})$ of elements of $S$ and a sequence $(a_1,\dots,a_k)$ of elements of $\N$ such that 
\[w = v_1 1^{a_1} v_2 1^{a_2} \cdots v_k 1^{a_k} v_{k+1}.\]
We say that \textbf{every word is used} in this building if $\{v_1,\dots,v_{k+1}\}=S$. A finite word $w$ is \textbf{ built from $S$} if there is a building of $w$ form $S$ with sequences $(v_1,\ldots,v_k,v_{k+1})$ in $S$ and $(a_1,\dots,a_k)$ in $\N$.
A finite word $w$ is \textbf{ built from $S$ starting with} $u$ if there is a building of $w$ form $S$ with sequences $(v_1,\ldots,v_k,v_{k+1})$ and $(a_1,\dots,a_k)$ such that $v_1=u$. These notions are extended to infinite words in a similar way.
\end{defn}\

\begin{example} \rm{
Consider the finite word
\[w =  01010101010\]
We first note that $w$ is built from $S=\{0\}$; in fact, every word in $\mathcal F$, and very infinite word starting with 0, is built from $S=\{0\}$. However, $w$ is also built from $S=\{010\}$, though it is not built from, for example, $S=\{00\}$.
The finite word 
\[u =  00101010101010
\]
is not built from any one-element set except $S=\{0\}$, but it is built from the two-element set $S=\{00, 010\}$.

The infinite Chac\'on word (defined in detail in Example~\ref{Chaconex})
\[C= 0010 0010 1 0010 0010 0010 1 0010 1 0010 0010 1 0010 \cdots\]
is built from $S=\{ 0010\}$, among other one-element sets.

}\end{example}

\begin{defn} \label{At Most Rank n Word}
An infinite word $V\in\{ 0,1 \}^\mathbb{N}$ \textbf{has an at most spacer rank-$\mathbf{n}$ construction}, 
if there exists an infinite sequence $(S_i)_{i\in\N}$, where each $S_i=\{v_{i, 1}, v_{i, 2}, \cdots, v_{i, n_i} \}$ is a set in $\mathcal{F}$ of $n_i$ words, with $1\leq n_i\leq n$, that is defined inductively by
\begin{align*}
&v_{0,j} = 0 \text{ for all } 1 \leq j \leq n_0\\
&v_{i+1, 1} \text{ is built from } S_i \text{ starting with } v_{i,1},\\
&v_{i+1, j} \text{ is built from }S_i , \text{ for all } 2 \leq j \leq n_i\leq n,
\end{align*}
and such that $V \upharpoonright |v_{i,1}| = v_{i,1}$ for all $i \in \mathbb{N}$. We then write $V = \lim_{i\rightarrow\infty} v_{i, 1}$.
For $n>1$, an infinite word $V\in\{ 0,1 \}^\mathbb{N}$ has a \textbf{spacer rank}-$\mathbf{n}$ \textbf{construction} if it has an at most spacer rank-$n$ construction and not an at most spacer rank-$(n-1)$ construction, and it is of \textbf{rank one} if it has a spacer rank-one construction. (Spacer rank-one coincides with the standard rank-one notion.) We call the set $S_{i}$ the $i$th \textbf{level} of the construction.
\end{defn}

We note that $V \upharpoonright |v_{i,1}| = v_{i,1}$ implies that $v_{i+1,1} \upharpoonright |v_{i,1}| = v_{i,1}$ for any $i \geq 0$.
Also, from Definition~\ref{D:built} it follows that the length of each $v_{i,j}$ in Definition~\ref{At Most Rank n Word} increases to infinity.
In addition, one can verify that if a word has a spacer rank $n$ construction, then it has a spacer rank-$(n+1)$ construction.

\begin{example} \label{Chaconex}\rm{
A well-known example of a word with a rank-one construction is the Chac\'on sequence or word, see e.g. \cite{Ju77}. We define the \textbf{Chac\'on sequence} $C$ as the limit $\lim_{i\rightarrow\infty} c_{i,1}$ where 

\begin{itemize}
\item[] $c_{0,1} = 0,$
\item[] $c_{i+1,1} = c_{i,1} c_{i,1} 1 c_{i,1}$ for all $i \in\N.$
\end{itemize}

Define the word $D$ by
\[D =  0C
\]
We give a spacer rank-two construction for $D$.
Define
\begin{itemize}
    \item[] $v_{0, 1} = 0$ and $v_{0, 2} = 0,$
     \item[] $v_{1, 1} = v_{0,1}v_{0,2} v_{0,2}1v_{0,2}$ and $v_{1, 2} = v_{0,2}v_{0,2}1v_{0,2}$,
    \item[] $v_{i+1, 1} = v_{i, 1} v_{i,2} 1 v_{i, 2}$ and $v_{i+1, 2} = v_{i, 2}v_{i,2}1v_{i,2}$,{ for }$i\geq 1$.
\end{itemize}
The word $D$ cannot have a rank-one construction. 
Suppose $D$ had a rank-one construction with levels $(w_{i,1})_{i\in\mathbb{N}}$; then $V=\lim_{i\to\infty} w_{i,1}$ and the $1$'s appear in $w_{i,1}$ only as single $1$'s. Consider a sufficiently large $i$ such that $w_{i,1}$ has length greater than 4, so $w_{i,1}=0001u$ for some finite word $u$ that starts and ends with $0$. One can check that there are no occurrences of four consecutive $0$'s in $D$. It follows that for any $j>i$, it must happen that
$$ w_{j,1}=w_{i,1}1w_{i,1}1\cdots 1w_{i,1}. $$
It follows that $D$, so $C$ is eventually periodic (i.e., $C$ is of the form $C=uvvv\cdots$ for some finite words $u,v$), a contradiction. So $D$ is not rank one.
}\end{example}

\begin{example} \rm{
The \textbf{Phouhet-Thue-Morse sequence}, or simply the \textbf{Morse sequence} or \textbf{Morse word}, denoted as $M$, has several definitions. A simple way to construct this sequence is by an inductive process where we start with 0 and then append its complement, so we obtain 01, and continue by appending the complement of 01 to obtain 0110, and then 01101001, etc. (A definition using substitutions is given in Section~\ref{S:sub}.) We now see that the Morse sequence has an at most spacer rank-two construction. Consider the following definition of $M_{i,j}$:
\begin{itemize}
    \item[] $M_{0, 1} = 0$ and $M_{0, 2} = 0;$
     \item[] $M_{i+1, 1} = M_{i, 1} 11 M_{i, 2}$ and $M_{i+1, 2} = M_{i, 2} M_{i, 1}$ if $i$ is even, for $i\geq 0$;
    \item[] $M_{i+1, 1} = M_{i, 1} 1 M_{i, 2}$ and $M_{i+1, 2} = M_{i, 2} 1 M_{i, 1}$ if $i$ is odd, for $i\geq 1$.
   
\end{itemize}

Then it is simple to check that $M \upharpoonright |M_{i,1}| = M_{i,1}$ for all $i \in \mathbb{N}$ and $M = \lim_{i\rightarrow\infty} M_{i, 1}$. Thus $M$ has an at most spacer rank-two construction. It follows from work of from del Junco \cite{Ju77} that the Morse system is not rank one; a different and independent argument is given as a consequence of {Corollary \ref{C:notrankone}}. Thus, the Morse sequence has a spacer rank-two construction.
}\end{example}

\begin{example} \rm{
We give a spacer rank-two construction for the following sequence
\[Q =  00101010\ldots
\]
Define
\begin{itemize}
    \item[] $v_{0, 1} = 0$ and $v_{0, 2} = 0,$
     \item[] $v_{1, 1} = v_{0,1} v_{0,2}$ and $v_{1, 2} = v_{0,2}1v_{0,2}$,
    \item[] $v_{i+1, 1} = v_{i, 1} 1 v_{i, 2}$ and $v_{i+1, 2} = v_{i, 2}1v_{i,2}$,{ for }$i\geq 1$.
\end{itemize}
The word $Q$ cannot have a rank-one construction. If $Q$ had a rank-one construction with levels $(w_{i,1})_{i\in\N}$,  since $V=\lim_{i\to\infty}w_{i,1}$, the word $w_{i,1}$ would have to start with 00, but as there are no other occurrences of 00 in $Q$, $w_{i,1}$ could not build $w_{i+1,1}$, a contradiction. 
}\end{example}

While the words $D$ and $Q$ have a spacer rank-two construction (and are not rank one), one could argue that ``essentially" they have a rank-one construction, and in fact in Section~\ref{S:finiterank} we will  see that $D$ and $Q$ define the same system, which is a rank-one system; however, $M$ does not define a rank-one system. This  motivates the following  Definition \ref{proper rank}.

\begin{defn} \label{proper rank}
An infinite word $V\in\{ 0,1 \}^\mathbb{N}$ \textbf{has a proper spacer rank-$\mathbf{n}$ construction} if it does not have an at most spacer rank-$(n-1)$ construction (when $n>1$) and there exists an infinite sequence $(S_i)_{i\in\N}$, where $S_0=\{0\}$ and for $i>1$  \[S_i=\{v_{i, 1}, v_{i, 2}, \cdots, v_{i, n} \}\] is a set in $\mathcal{F}$ of $n$ words   that is defined inductively by  
\begin{align*}
&v_{0,j} = 0 \text{ for all } 1 \leq j \leq n,\\
&v_{i+1, 1} \text{ is built from }   S_i \text{ starting with }v_{i,1},\\
&v_{i+1, j} \text{ is built from } S_i, \text{ for all } 2 \leq j \leq n.\\
\end{align*}
In addition, every word of $S_i$ is used in the building of $v_{i+1,j}$ { for all } $1 \leq j \leq n$,
and $V \upharpoonright |v_{i,1}| = v_{i,1}$ for all $i \in \mathbb{N}$. 
\end{defn}

Every rank-one construction is a proper rank-one construction. The Morse word $M$ has a proper spacer rank-two construction, while the infinite word $Q$ does not have a proper spacer rank-$n$ construction for any $n \geq 1$. We note that if a word $V$ has a proper spacer rank-$n$ construction, then it is \textbf{recurrent}, i.e.,  every finite subword of $V$ appears in $V$ infinitely often. The word $Q$ is not recurrent as $00$ appears only once. \\

We note that it may happen that a spacer rank construction starts to have the proper spacer rank property after finitely many steps of the construction, or in an infinite subsequence; if this is the case we will typically assume that the proper spacer rank sub-construction has been chosen.


We now give examples, for each $n\geq 1$, of infinite words that have a spacer rank-$n$ construction, and also examples of words that do not have a spacer rank-$n$ construction for any $n \in \bn$. We note that for any infinite word $V$, the word $1V$ does not have a spacer rank-$n$ construction for any $n\geq 1$, though it would not be reasonable to call it of unbounded spacer rank. We thus introduce the following definition.

\begin{defn} \label{unbounded rank}
An infinite word $V\in\{ 0,1 \}^\mathbb{N}$ \textbf{has a spacer rank construction} if there exists a sequence $(n_i)_{i\geq 0}$ and a sequence 
$(S_i)_{i \in \mathbb{N}}$ where each $S_i$ is a finite set of finite words in $\mathcal{F}$ of the form
\[S_i=\{v_{i,1}, v_{i,2},\ldots, v_{i,n_i}\},\] defined inductively by
\begin{align*}
&v_{0,j} = 0 \text{ for all } 1 \leq j \leq n_0,\\
&v_{i+1, 1} \text{ is built from }    S_i \text{ starting with }v_{i,1},\\
&v_{i+1, j} \text{ is built from }  S_i, \text{ for all } 2 \leq j \leq n,
\end{align*}
and such that $V \upharpoonright |v_{i,1}| = v_{i,1}$ for all $i \in \mathbb{N}$. 
A word is said to  have \textbf{unbounded spacer rank} if it has a spacer rank construction but does not have a spacer rank-$n$ construction for any $n\geq 1$.  
\end{defn}

In this definition the cardinality of the sets $S_i$ is not uniformly bounded, so it follows that if a word has a spacer rank-$n$ construction then it has a  spacer rank construction, but there are words that have no  spacer rank construction such as the word $1V$.
We first construct words of proper finite spacer rank.

\begin{prop}\label{P:allfiniterank} For each $n\in\N^+$ there is an infinite word $V$ which has a proper spacer rank-$(n+1)$ construction and no spacer rank-$n$ construction.
\end{prop}

\begin{proof} Let $F_{n,k}$, $k\geq 1$, enumerate all sets of $n$ words $w_1,\dots, w_n$, where each $w_i$ starts and ends with $0$, and $|w_i|>1$. (We allow repetitions in the words $w_1, \dots, w_n$.)
Let $M_{n,k}$ be larger than $|w_1|+\cdots+|w_n|$ if $F_{n,k}=\{w_1,\dots, w_n\}$.

Define a proper spacer rank-$(n+1)$ construction by starting as follows. For $1\leq j\leq n+1$, let
$$ v_{0,j}=0 $$
and
$$ v_{1,j}=01^j0. $$
By induction on $k$, assume that $v_{k,j}$, for $1\leq j\leq n+1$, have been defined. We define $v_{k+1,j}$ for $1\leq j\leq n+1$. Consider two cases:
\begin{quote}
\begin{enumerate}
\item[Case 1.] Every $v_{k, j}$, $1\leq j\leq n+1$, is built from $F_{n,k}$.

    In this case we let
    $$ v_{k+1,j}=v_{k,j} $$
    for all $1\leq j\leq n$.
\item[Case 2.] There is $1\leq j_0\leq n+1$ such that $v_{k,j_0}$ is not built from $F_{n,k}$.

    Let $M$ be larger than $|v_{k,j}|$ for all $1\leq j\leq n+1$ and also larger than $M_{n, k}$. For each $1\leq j\leq n+1$, let
    $$ v_{k+1,j}=u_01^{n_1}u_11^{n_2}\cdots 1^{n_\ell}u_{\ell} $$
    where
    \begin{itemize}
    \item $\ell\geq n$
    \item $u_0=v_{k,j}$
    \item $\{u_p\,:\, 0\leq p\leq \ell\}=\{u_p\,:\, 1\leq p\leq \ell-1\}=\{v_{k,j}\,:\, 1\leq j\leq n+1\}$
    \item for all $1\leq p\leq \ell$, $n_p>M$.
    \end{itemize}
\end{enumerate}
\end{quote}
This finishes  the definition of the word. To finally obtain a proper  spacer rank-$(n+1)$ construction we only need to rearrange the stages of the construction to omit any step where Case 1 happens.

We prove that $V$ does not have an at most spacer rank-$n$ construction. If $V$ had an at most spacer rank-$n$ construction, then  there would exist a $k\geq 1$ such that $F_{n,k}=\{w_1,\dots, w_n\}$ appears in such a construction and we have $|w_1|, \dots, |w_n|> n+3$.

Consider the $k$-th step of the above construction. Suppose first that Case 1 happens, i.e., every $v_{k,j}$, $1\leq j\leq n+1$, is built from $F_{n,k}$. Then each $v_{k,j}$ has some $u_j\in F_{n,k}$ as its initial segment. Note that by our definition, each $v_{k,j}$ also has $v_{1,j}=01^j0$ as its initial segment. Thus, since $|u_j|>n+3$, $u_j$ must have $01^j0$ as its initial segment. This implies that there are at least $n+1$ many distinct elements in $F_{n,k}$, a contradiction.

Thus Case 2 must have happened in the $k$-th step of the above construction, i.e., there is $1\leq j_0\leq n+1$ such that $v_{k,j_0}$ is not built from $F_{n,k}$. Suppose
$$ v_{k+1,1}=v_{k,1}1^{n_1}u_11^{n_2}\cdots 1^{n_{\ell}}u_{\ell}. $$
By our construction there is $1\leq p_0\leq \ell-1$ such that $u_{p_0}=v_{k,j_0}$. Since $v_{k+1,1}$ is an initial segment of $V$ and $V$ is built from $F_{n,k}$, and since $n_p>M$ for all $1\leq p\leq \ell$, it follows that each $u_p$, $1\leq p\leq \ell-1$, must be built from $F_{n,k}$. This contradicts our assumption about $u_{p_0}=v_{k,j_0}$, showing that $V$ does not have a spacer rank-$n$ construction.
\end{proof}

We now construct infinite words starting with 0 with no spacer rank construction. Recall that the \textbf{complexity} of an infinite word $V$, $P_n(V)$, is  defined to be the number of subwords or factors  in $V$ of length $n$. A word $B$ containing every finite binary word has  complexity $P_n(B) = 2^n$; we say such a word is a a \textbf{full complexity word}. We show below that a full complexity word does not have a  spacer rank construction, and  in Example~\ref{E:polyunranked} we construct a word that does not have a spacer rank construction but has polynomial  complexity.

\noindent The following will show that a full complexity  word does not have a spacer rank construction.

\begin{lem}\label{L:unranked}
If $K$ is a word  containing $1^k01^k$ for all $k \in \bn$, then $K$ does not have a  spacer rank construction.
\end{lem}

\begin{proof}
 Suppose that  $K$ has a  spacer rank construction, i.e.,
there is a doubly-indexed sequence of words $(v_{i,j})_{i\in\N, j\leq m_i}$ such that each $v_{i+1, j}$ is built from $$S_i=\{v_{i,j}\,:\, j\leq m_i\}.$$
From Definition~\ref{D:built} it follows that  for every $i>0$, all words in $S_i$ have length greater than 1, and  in fact, each of them must contain at least two $0$s. The spacer ranked construction also guarantees that for each $i\in\N$, $K$ is built from $S_i$.

Now fix any $i>0$; we derive a contradiction. Let $M$ be an upper bound for the lengths of words in $S_i$. Consider $$w=1^{M+1}01^{M+1}.$$
By assumption, $w$ occurs in $K$. Assume the $0$ in the middle of $w$ occurs at position $p$ in $K$. Then since $K$ can be written as
$$ K=w_11^{a_1}w_21^{a_2}\cdots, $$
where for all $n$, $w_n\in S_i$ and $a_n\in \mathbb{N}$, the position $p$ falls into some $w_n$ as presented above. 
Since $|w_n|<M$ and $w_n$ contains at least two $0$s, we get a contradiction.

\end{proof}

\begin{cor}
Full complexity words do not have a spacer rank construction. 
\end{cor}

We conclude this section by showing that there are words with a spacer rank construction that do not have a finite spacer rank construction.

\begin{prop}\label{infiniterank}
There exists a word of unbounded spacer rank.
\end{prop}

\begin{proof}
Let $F_i$, $i\geq 1$, enumerate all finite subsets of $\mathcal F$. Let $(a_i)_{i\geq 1}$ be a strictly increasing sequence such that
 \[a_i \geq 2 + \sum_{v\in F_i} |v|.\] Let $(n_i)_{i\geq 1}$ be inductively defined as
 $$n_1 = 3, \mbox{ and } n_{i+1} = n_i! \mbox{ for $i\geq 1$.} $$
 Consider the following ranked construction:
\[v_{0,1} = 0, \]
\[ v_{1,1} = 010,\
 v_{1,2} = 01^20,\
  v_{1,3} = 01^30.\]
For $i\geq 1$, suppose $v_{i,j}$, $1\leq j\leq n_i$, have been defined. We define $v_{i+1, j}$ for $1\leq j\leq n_{i+1}$. Note that $n_{i+1}=n_i!=|\mbox{Sym}(\{1, 2,\dots, n_i\})|$ (where $Sym$ stands for the symmetric group). We let $f_i$ be a bijection from $\{1, 2,\cdots , n_{i+1}\}$ to $\mbox{Sym}(\{1, 2,\dots, n_i\})$ such that $f_i(1)(j)=j$ for all $1\leq j\leq n_i$. Define, for $1\leq j\leq n_{i+1}$,
   \[v_{i+1,j} =
v_{i,f_i(j)(1)}1^{a_i}v_{i,f_i(j)(2)}1^{a_i}\cdots 1^{a_i}v_{i,f_i(j)(n_i)}.\]
Note that $v_{i,1}$ is an initial segment of $v_{i+1,1}$. We thus obtain a  spacer ranked construction for the
word $V = \lim_i v_{i,1}$.

By a standard induction, we have that 
\begin{enumerate}
\item[(*)] for any $i$ and $1\leq j\neq j'\leq n_i$, $v_{i,j}$ is not an initial segment of $v_{i,j'}$, and in particular $v_{i,j}\neq v_{i,j'}$;
\item[(**)] for any $i<i'$ and $1\leq j\leq n_i$ there is some $1\leq j'\leq n_{i'}$ such that $v_{i', j'}$ has $v_{i,j}$ as an initial segment.
\end{enumerate}

We claim that $V$ is not of finite spacer rank, thus it is a word of unbounded spacer rank. Assume toward a contradiction that $V$ has an at most spacer rank-$n$ construction, for some $n$. Let $(w_{r,s})_{r\geq 1, 1\leq s\leq k_r\leq n}$ be the levels of a spacer rank-$n$ construction of $V$.
Let $i_0$ be sufficiently large such that $n_{i_0}>n$. Let $r_0$ be sufficiently large such that for all $r\geq r_0$, for all $1\leq s\leq k_{r}\leq n$ and for all $1\leq j\leq n_{i_0}$,
 $$ |w_{r,s}|>|v_{i_0, j}|. $$
Let $i_1>i_0$ be such that $F_{i_1}=\{w_{r_1,s}\,:\, 1\leq s\leq k_{r_1}\}$ for some $r_1\geq r_0$.

Since $V$ has $v_{i_1+1, 1}$ as an initial segment, $V$ is built from $F_{i_1}$, and $a_{i_1}>|w_{r_1,s}|$ for all $1\leq s\leq k_{r_1}$, we conclude that for all $1\leq j\leq n_{i_1}$,
$$ v_{i_1, f_{i_1}(1)(j)} \mbox{ is built from $F_{i_1}$.} $$
In particular, for any $1\leq j\leq n_{i_1}$, $v_{i_1, j}$ is built from $F_{i_1}$.
By (**) we have that for any $1\leq j\leq n_{i_0}$ there is $1\leq j'\leq n_{i_1}$ such that $v_{i_0, j}$ is an initial segment of $v_{i_1, j'}$. Since $|w_{r_1, s}|>|v_{i_0, j}|$ for all $1\leq s\leq k_{r_1}$ and $1\leq j\leq n_{i_0}$, we conclude that for any $1\leq j\leq n_{i_0}$ there is some $1\leq s\leq k_{r_1}$ such that
$v_{i_0,j}$ is an initial segment of $w_{r_1, s}$. By (*) this implies that
$$ |\{v_{i_0,j}\,:\, 1\leq j\leq n_{i_0}\}|\leq |\{w_{r_1,s}\, :\, 1\leq s\leq k_{r_1}\}|. $$
 Hence $n\geq k_{r_1}\geq n_{i_0}$, contradicting our assumption that $n_{i_0}>n$.
\end{proof}

\section{Spacer Rank Constructions for Fixed Points of Substitutions and Generalized Morse Sequences}\label{S:sub}

In this section we explore sequences that are fixed points of substitutions and introduce a criterion for aperiodic substitutions having spacer rank greater than one; we refer to \cite{Fo02,Qu10} for background on substitutions. We first investigate the spacer rank of substitutions. 

Given a function $  \zeta: \{0, 1\} \to  \{0, 1\}^{<\N}$ and a sequence $w = w_0w_1w_2 \dots \in \{0, 1\}^\mathbb{N}$, we let  $\zeta(w)$ denote  $\zeta(w_0)\zeta(w_1)\zeta(w_2) \dots$.

\begin{defn}
A \textbf{substitution} is a function $\zeta: \{0, 1\} \to  \{0, 1\}^{<\N}$. We assume $\zeta(0)$  starts with 0 and has  length greater than 1 (so $\lim_{n\to\infty}|\zeta^n(0)|=\infty$), and we also assume that $\lim_{n\to\infty}|\zeta^n(1)|=\infty$.   It follows that there  is a sequence $u$ such that $\zeta(u) = u$, a \textbf{fixed point} of the substitution.  So there is a unique $u$ starting with $0$  and we can write
$u =\lim_{n\to\infty}\zeta^n(0)$ \cite{Fo02,Qu10}. All the sequences  we consider that are fixed points of substitutions  start with 0.
\end{defn}

\begin{example} \rm{
The Morse sequence is a fixed point of the substitution $0 \mapsto 01$ and $1 \mapsto 10$. 
}\end{example}

\begin{prop}
A sequence that is a fixed point of a substitution has an at most spacer rank two construction.
\end{prop}

\begin{proof}
Let $q$ denote a  sequence that is a fixed point of the substitution $\zeta$. For $k$ a positive integer, let $\zeta^{2^k}(0) = v_{k} 1^{x_{k}}$ where $v_k$ starts and ends with $0$ and $\zeta^{2^k}(1) = 1^{y_{k}} w_{k} 1^{z_k}$ where $x_k, y_k, z_k$ are nonnegative integers. Represent $\zeta$ as $0 \mapsto a_0a_1a_2 \dots a_n$ with $a_0 = 0$ and $1 \mapsto b_0b_1b_2\dots b_m$. Suppose $m > 0$ and $n > 0$ and also that there is some $j$ such that $b_j = 0$. We claim the following:

\begin{equation}\label{E:estimate}
|v_{r + 1}| > |v_r|\text{ and }|w_{r + 1}| > |w_r|.
\end{equation}

We first prove \eqref{E:estimate}. 
By replacing $\zeta^{2^r}$ with an arbitrary substitution $\tilde{\zeta}$ (or simply by an induction argument), we may assume that $r = 1$. Let $n'$ be the maximal integer such that $1 \le n' \le n$ and for $i > n'$, $a_i = 1$ and $a_{n' - 1} = 0$. Such an $n'$ must exist since the first digit $a_0 = 0$. Let $\zeta^2(0) = 0c_1 c_2 \dots c_k$ and $\zeta^2(1) = d_0d_1d_2 \dots d_p$. Choose $k'$ to be the maximal integer such that $1 \le k' \le k$ and for $i > k'$, $c_i = 1$ and $c_{k' - 1} = 0$. We show that $k' > n'$. First, observe that $k' \ge n'$ since the first digit of $\zeta(0)$ is $0$ so $0c_1 \dots c_n = 0a_1, \dots, a_n$. Then $k' > n'$ because there is some $b_j$ such that $b_j = 0$ so even if $a_i = 1$ for all $i \ge 0$, $c_\ell = 0$ for some $\ell > n$. Therefore, as $k' \ge \ell$, we must have $k' > n'$ and thus $|v_{2}| > |v_1|$. 

Suppose $b_0 = b_m = 1$. Let $m_1$ be the least integer greater than $0$ such that $b_{m_1} = 0$ and $m_2$ be the greatest integer less than $m$ such that $b_{m_2} = 0$. Let $p_1$ be the least integer greater than $0$ such that $d_{p_1} = 0$ and $p_2$ be the greatest integer less than $p$ such that $d_{p_2} = 0$. Observe that $p_2 - p_1 > m_2 - m_1$. Indeed, $p_2 - p_1 \ge m_2 - m_1$ since the first digit is a $1$ so $p_1 = m_1$ and $p_2 \ge m_2$. Because the first and last digits of $b$ is a $1$ and the middle digit is a $0$, we know that $m \ge 2$. Consequently, $p_1 = m_1$ and $p_2 > m_2$ so $p_2 - p_1 > m_2 - m_1$ so $|w_{2}| > |w_1|$. 

Next, suppose $b_0 = 0$. Let $m'$ be the greatest integer such that $b_{m'} = 0$ and $p'$ the greatest integer such that $d_{p'} = 0$. Note once again that $p' \ge m'$ since $\zeta(0)$ starts with a $0$ and $p' > m'$ since $m > 0$: the digit $b_1$ must be either a $0$ or a $1$ and if it were a $0$, then $d_0 = d_1 = 0$ and $d_{m + 1} = d_{m + 2} = 0$ so $p' > m \ge m'$ and $|w_2| > |w_1|$. 

If $b_m = 0$, let $m''$ be the least integer such that $b_{m''} = 0$ and $p''$ be the least integer such that $d_{p''} = 0$. A similar argument as above shows that $p'' > m''$ and thus $|w_2| > |w_1|$.

Notice that $v_{k + 1}$ and $w_{k + 1}$ can be built from $v_k$ and $w_k$. This is because $\zeta^{2^{k + 1}}(0)$ and $\zeta^{2^{k + 1}}(1)$ can be built from $\zeta^{2^{k}}(0)$ and $\zeta^{2^{k}}(1)$ and because $\zeta^{2^{k + 1}} = \zeta^{2^k} \circ \zeta^{2^k}$, it follows that $x_{k + 1}, y_{k + 1}, z_{k + 1} \in \{x_k, y_k, z_k, 0\}$. Hence, $v_{k + 1}$ and $w_{k + 1}$ can be built from $v_k$ and $w_k$. In addition, $q$ starts with $v_k$ since the first digit of $q$ is $0$. Hence $v_{k}$ and $w_k$ are a sequence of words of increasing length that build $q$, so $q$ is spacer rank two if there exists $j$ such that $b_j = 0$ and if $m, n > 0$. 

If $n = 0$, then the sequence is trivial and thus rank one. If $m = 0$, suppose $b_0 = 0$. Then the substitution $\zeta^2$ satisfies the condition that $|\zeta^2(0)| > 2$ and $|\zeta^2(1)| > 2$ and that there is some symbol in $\zeta^2(1)$ that is $0$, and we are in the case of $m, n > 0$ and there is some $j$ with $b_j = 0$. This leaves us with the last remaining case that $b_i = 1$. In this case, we show that $q$ is rank one. If $a_1 \dots a_N$ are all $1$'s, then $q$ is simply $011111\dots$. If there exists some $a_i = 0$ for $i \ge 1$, then we claim that $|v_2| > |v_1|$. Once again, $|v_2| \ge |v_1|$ since the first digit of $v_1$ is $0$ $|v_2| > |v_1|$ since the substitution of the second $0$ contains a $0$ and that $0$ is in a further place than $|v_1|$. Since $w_k$ are all empty, $v_{k + 1}$ are all built from $v_k$ and $q$ is rank one.
\end{proof}

\begin{example} \rm{
Let $\zeta$ denote the  \textbf{Fibonacci substitution} $0 \mapsto 01$ and $1 \mapsto 0$. Let $W =\lim_{n \to \infty} \zeta^n(0)$. As $W$ is a fixed point of  a substitution, it has an at most spacer rank two construction. We will show below that the dynamical system associated to the Fibonacci sequence is not rank one, thus showing it has a spacer rank two construction.
}\end{example}

\begin{example} \rm{
Let $\zeta$ denote the \textbf{Cantor substitution} $0 \mapsto 010$ and $1 \mapsto 111$ and the sequence $W = \lim_{n \to \infty} \zeta^n(0)$. As remarked in the proof of the above proposition, since $\zeta(1)$ has no zeroes in it, the Cantor sequence is rank one.
}\end{example}

\subsection{Periodic words that are fixed points of  substitutions}

We start with a few lemmas about when finite words must be periodic. 
The following lemma follows by induction and its proof is left to the reader.

\begin{lem}\label{lem1} Suppose $\alpha, \beta\in\{0,1\}^{<\N}$ satisfy $\alpha\beta=\beta\alpha$. Then there is $\gamma\in\{0,1\}^{<\N}$ and $m, n\in\mathbb{N}$ such that
$\alpha=\gamma^n$ and $\beta=\gamma^m$.
\end{lem}


\begin{cor}\label{lem2} Suppose $\alpha\in\{0,1\}^{<\N}$ is a subword of $\alpha\alpha$ whose occurrence does not coincide with either of the demonstrated copies of $\alpha$ in $\alpha\alpha$. Then there is $\beta$ and $n>1$ such that $\alpha=\beta^n$.
\end{cor}


\begin{lem}\label{lem3} Let $\alpha, \beta\in\{0,1\}^{<\N}$, $n\geq 2$ and $m\geq 1$. Suppose $|\alpha|>|\beta|$ but $\alpha^n$ is an initial segment (or an end segment) of $\beta^m$. Then there is $\gamma$ and $k, l\geq 1$ such that $\alpha=\gamma^k$ and $\beta=\gamma^l$.
\end{lem}

\begin{proof} First assume $\alpha^n$ is an initial segment of $\beta^m$. Let $\eta$ and $p\geq 1$ be such that $|\eta|<|\beta|$ and $\alpha=\beta^p\eta$. If $\eta$ is empty then there is nothing to prove. Otherwise, we have $\eta\beta=\beta\eta$. By Lemma~\ref{lem1}, there is $\gamma$ and $t, s$ such that $\eta=\gamma^t$ and $\beta=\gamma^s$. Then $\alpha=\beta^p\eta=\gamma^{ps+t}$. For the case when $\alpha^n$ is an end segment of $\beta^m$, reverse the order of the words and argue similarly.
\end{proof}

\begin{cor}\label{cor3.5} Let $\alpha, \beta\in\{0,1\}^{<\N}$, $n, m\geq 1$. Suppose $\alpha^n=\beta^m$. Then there is $\gamma$ and $k, l\geq 1$ such that $\alpha=\gamma^k$ and $\beta=\gamma^l$.
\end{cor}


\begin{defn}
Consider a periodic infinite word $V$. We say that a finite word $v$ is a \textbf{periodic building block} of $V$ if  $V=vvv\cdots$.
$v$ is called a \textbf {principal periodic building block} if $v$ is a periodic building block and for every periodic building block $u$ of $V$, there is some $k\geq 1$ with $u=v^k$. 
\end{defn}

The following lemma follows directly from Lemma~\ref{lem3}. 

\begin{lem}\label{lem4} If $\alpha$ and $\beta$ are two periodic building blocks of $V$, then there is a periodic building block $\gamma$ and $k, l\geq 1$ such that $\alpha=\gamma^k$ and $\beta=\gamma^l$.
\end{lem}


\begin{prop} Every periodic infinite word has a unique principal periodic building block.
\end{prop}

\begin{proof} Let $v$ be the shortest periodic building block of $V$. If $u$ is another periodic building block, then by Lemma~\ref{lem4} there is a periodic building block $\gamma$ and $k, l\geq 1$ such that $v=\gamma^k$ and $u=\gamma^l$. By the minimality of the length of $v$, we have $\gamma=v$ and $k=1$. Thus $u=v^l$, and $v$ is principal.
\end{proof}

\begin{defn}
We say $v$ is a \textbf{periodic building block} of a finite word $u$ if $u=v^k$ for some $k\geq 1$. We say a substitution $\zeta$ is \textbf{nontrivial} if in addition to our assumptions $\zeta(0)$ contains a $1$. 
\end{defn}


\begin{prop}\label{mainprop} Let $V$ be a periodic infinite word and $v$ be its principal periodic building block. Let $\zeta$ be a nontrivial substitution with $|\zeta(0)|\geq |v|$. Suppose $V=\lim_n \zeta^n(0)$. Then one of the following holds:
\begin{enumerate}
\item[\rm (i)] $v$ is a periodic building block of both $\zeta(0)$ and $\zeta(1)$.
\item[\rm (ii)] $v=01^t$ for some $t\geq 2$, $V=(01^t)^\infty$, $\zeta(0)=01^t0$ and $\zeta(1)=1$.
\item[\rm (iii)] $v=01$, $V=(01)^\infty$, $\zeta(0)=(01)^a0$ for some $a\geq 1$ and $\zeta(1)=1(01)^b$ for some $b\geq 0$.
\end{enumerate}
\end{prop}

\begin{proof}
Write
$$ V=a_1a_2\cdots $$
with $a_i\in\{0,1\}$ for all $i\geq 1$. In fact, $a_1=0$. Note that $\zeta(V)=V$, i.e., we can also write
$$ V=\zeta(a_1)\zeta(a_2)\cdots. $$

First suppose that $v$ is a periodic building block of $\zeta(0)$ but not of $\zeta(1)$. Then there exist $k\geq 0$ and $\alpha\in\{0,1\}^{<\N}$ with $0<|\alpha|<|v|$ and $\zeta(1)=v^k\alpha$.

By Lemma~\ref{lem2}, the starting position of any $\zeta(a_i)$ when $a_i=0$ must be one plus a multiple of $|v|$. This is because, otherwise the first copy of $v$ in $\zeta(a_i)$ would be a subword of $vv$ whose occurrence in $vv$ does not coincide with either of the demonstrated copies of $v$ in $vv$, and Lemma~\ref{lem2} gives a shorter building block of $V$ than $v$, contradicting the assumption that $v$ is principal. 

Since $V$ is periodic and starts with $0$, $V$ contains infinitely many $0$s. Thus the maximal blocks of $1$s in $V$ are finite. Consider a finite maximal block of $1$s, say it is an occurrence of $1^t$ for some $t\geq 1$. By comparison we get that $\zeta(1)^t=v^s$ for some $s\geq 1$. Thus Corollary~\ref{cor3.5} gives a periodic building block $\gamma$ of both $\zeta(1)$ and $v$. We must have $|\gamma|<|v|$, which contradicts the assumption that $v$ is principal. This completes the proof in the first case.

Next suppose that $v$ is a periodic building block of $\zeta(1)$ but not of $\zeta(0)$. Let $j$ be the least such that $a_j=1$. If the starting position of $\zeta(a_j)$ is not one plus a multiple of $|v|$, then we apply Lemma~\ref{lem2} to get a shorter periodic building block than $v$, contradicting the principality of $v$. Thus the starting position of $\zeta(a_j)$ is one plus a multiple of $|v|$, which implies that $\zeta(0)^{j-1}=v^m$ for some $m\geq 1$. Thus by Corollary~\ref{cor3.5} we get a shorter periodic block than $v$, again contradicting the principality of $v$. This completes the proof in the second case.

Finally suppose $v$ is a periodic building block of neither $\zeta(0)$ nor $\zeta(1)$. Since $V$ is periodic, $V$ contains infinitely many $0$s as well as infinitely many $1$s. Since $|\zeta(0)|>|v|$, we get that for any $i\geq 1$, if $a_i=0$, then the starting position of $\zeta(a_i)$ in $V$ is one plus a multiple of $|v|$. This is because, otherwise we have that the first copy of $v$ in $\zeta(a_i)$ is a subword of $vv$ whose occurrence does not coincide with either of the demonstrated copies of $v$ in $vv$, and by Lemma~\ref{lem2} we get a shorter periodic building block of $V$, contradicting the principality of $v$. In particular, we conclude that $V$ does not contain $00$. 

Suppose $V$ contains $11$. Consider an arbitrary maximal block of $1$s in $V$, say it is of the form $1^t$ with $t\geq 2$. By comparison we get that there are $k, l$ such that $v^k=\zeta(0)\zeta(1)^t$. Now if there are $t<t'$ such that $1^t$ and $1^{t'}$ are both maximal blocks of $1$s, then by comparing $v^k=\zeta(0)\zeta(1)^t$ with $v^{k'}=\zeta(0)\zeta(1)^{t'}$, we get that $v^{k'-k}=\zeta(1)^{t'-t}$, and Corollary~\ref{cor3.5} gives a shorter periodic building block than $v$, contradicting the principality of $v$. Thus we conclude that there is a unique $t\geq 2$ as the length of all maximal blocks of $1$s in $V$, and that $V=(01^t)^\infty$. As $v$ is the principal periodic building block of $V$, we must have $v=01^t$.

Now if $|\zeta(1)^t|=t|\zeta(1)|\geq 2|v|=|v^2|$, then $v^2$ is an end segment of $\zeta(1)^t$, and by Lemma~\ref{lem3}, either $v$ is a periodic building block of $\zeta(1)$, contradicting our case assumption, or we obtain a shorter periodic building block than $v$, contradicting the principality of $v$. Thus we have $|\zeta(1)^t|<|v^2|$ and $|\zeta(0)|>|v|$. Since $\zeta(1)^t$ is an end segment of $v^2$, and noting that $v^2=01^t01^t$, we conclude that $\zeta(1)^t$ either contains exactly one $0$, which is absurd since $t\geq 2$, or $\zeta(1)^t$ contains no $0$. In this last situtation, we have that $\zeta(0)$ contains at least two $0$s and $\zeta(1)=1^b$ for some $b\geq 1$. It is easy to see that we must have $\zeta(0)=01^t0$ and $\zeta(1)=1$.

We are left with the case that $V$ does not contain $11$. In this case we obviously have $V=(01)^\infty$, $v=01$, $\zeta(0)=(01)^a0$ for some $a\geq 1$ and $\zeta(1)=1(01)^b$ for some $b\geq 0$.
\end{proof}

We next give a lemma that begins connecting our previous lemma with the Euclidean pairs of words that we are about to define.

\begin{lem}\label{lem6} Let $\zeta$ be a nontrivial substitution and $a\in \{0,1\}$. Then $\zeta(a)$ is an initial  (respectively, end) segment of $\zeta(a^c)$ if and only if  $\zeta^2(a)$ is an initial (respectively, end) segment of $\zeta^2(a^c)$.
\end{lem}

\begin{proof} We prove the case where $a=0$ and $\zeta(0)$ is an initial segment of $\zeta(1)$. The other cases are similar. Suppose $\zeta(0)$ is an initial segment of $\zeta(1)$. Since $\zeta^2(0)=\zeta(\zeta(0))$ and $\zeta^2(1)=\zeta(\zeta(1))$, we have that $\zeta^2(0)$ is an initial segment of $\zeta^2(1)$. Conversely, suppose $\zeta(0)$ is not an initial segment of $\zeta(1)$ but $|\zeta(0)|\leq |\zeta(1)|$. Then there are $\alpha\in\{0,1\}^{<\N}$ and $b\neq c\in\{0,1\}$ such that $\alpha b$ is an intial segment of $\zeta(0)$ and $\alpha c$ is an initial segment of $\zeta(1)$. Then $\zeta(\alpha)\zeta(b)$ is an initial segment of $\zeta^2(0)$ and $\zeta(\alpha)\zeta(c)$ is an intial segment of $\zeta^2(1)$. Since $\zeta(\alpha)\zeta(b)$ and $\zeta(\alpha)\zeta(c)$ are not initial segments of one another, we have that $\zeta^2(0)$ is not an intial segment of $\zeta^2(1)$.
\end{proof}

\begin{defn}
Let $\alpha, \beta\in\{0,1\}^{<\N}$. We call $(\alpha, \beta)$ a \textbf{Euclidean pair} if there is $\gamma\in\{0,1\}^{<\N}$ and $k, l\geq 1$ such that
$\alpha=\gamma^k$ and $\beta=\gamma^l$. A nontrivial substitution $\zeta$ is called \text{Euclidean} if $(\zeta(0), \zeta(1))$ is a Euclidean pair. 
\end{defn}

The following lemma follows from the definitions.

\begin{lem}\label{lem7} Let $\alpha, \beta\in\{0,1\}^{<\N}$. Suppose $|\alpha|>|\beta|$. Then $(\alpha, \beta)$ is a Euclidean pair if and only if  there is $k\geq 1$ and $\gamma\in\{0,1\}^{<\N}$ such that $\alpha=\beta^k\gamma$ and $(\beta,\gamma)$ is a Euclidean pair.
\end{lem}

The above lemma justifies the terminology. If $(\alpha, \beta)$ is a Euclidean pair, then we can perform the ``Euclidean algorithm" suggested by the lemma to arrive at a $\gamma$ that is a common ``factor" of $\alpha$ and $\beta$. Conversely, if the Euclidean algorithm is successfully performed, then $(\alpha, \beta)$ is a Euclidean pair. 

\begin{lem}\label{lem8} Let $\zeta$ be a nontrivial substitution. Then $\zeta$ is Euclidean if and only if  $\zeta^2$ is Euclidean.
\end{lem}
\begin{proof} The forward direction follows directly from the definition. For the converse, suppose $\zeta^2$ is Euclidean. When  $\zeta^2(0)$ is an initial and an end segment of $\zeta^2(1)$ . 
 (the proof is
similar when $\zeta^2(1)$ is an initial and an end segment of $\zeta^2(0)$), $\zeta(0)$ is an initial and
an end segment of $\zeta(1)$), and there exist $m, n\in\N^+$, such that $\zeta((\zeta(0))^m) = \zeta((\zeta(1))^n)$. If
$\zeta(1)$ is not an initial segment of $(\zeta(0))^m$, then consider the first position $i$ such
that $\zeta(1)(i)\neq (\zeta(0))^m(i)$ (there exists such $i$ because $(\zeta(0))^m$ is not an initial
segment of $\zeta(1))$, $\zeta((\zeta(0))^m) = \zeta((\zeta(1))^n)$, then there must exist $m_1\in\N^+$ such
that $\zeta(1)$ is an initial segment of $\zeta(0^{m_1} 1)$, contradicting to our assumption. So
$\zeta(1)$ is an initial segment of $(\zeta(0))^m, \zeta(1) = (\zeta(0))^{m^\prime}
v$, $|v| < |\zeta(0)|$,$m^\prime\in\N^+$,
if $|v| = 0$ then $\zeta$ is Euclidean. If $|v|\neq $0, then $\zeta(0) = vv_1$ for some $v_1\in\{0,1\}^{<\N}$.
$\zeta(0)$ is an initial and end segment of $\zeta(1)$, so $v$ is an end segment of $\zeta(0)$. If
$v_1$ is not an initial segment of $\zeta(0)$, then $(\zeta(0))^{m^\prime+1 }$ is not an initial segment
of $(\zeta(1))^2$, consider the first position $j$ such that $(\zeta(0))^{m\prime+1}(j)\neq (\zeta(1))^2(j)$,
$\zeta((\zeta(0))^m) = \zeta((\zeta(1))^n)$, then there must exist $m_2\in\N^+$ such that $\zeta(1)\zeta(0)$ is
an initial segment of $\zeta(0^{m_2} 1)$, contradicting to our assumption. So $v_1$ is an initial
segment of $\zeta(0)$, then there exist $\gamma\in \{0,1\}^{<\N}, m_3,m_4\in\N, v =\gamma^{m_3} , \zeta(0) =\gamma^{m_4} $. It follows that $\zeta$ is Euclidean.

\end{proof}

It follows from this lemma that if $\zeta$ is Euclidean, then $\zeta^{2^k}$ is Euclidean for any $k$ which we can use to prove the next theorem which provides further classification of periodic substitutions.

\begin{thm}\label{periodic} Let $V\neq (01^t)^\infty$, for any $t\geq 1$, be a periodic word and $v$ be its principal periodic building block. Let $\zeta$ be a nontrivial substitution with $V=\lim_n\zeta^n(0)$. Then $v$ is a periodic building block of both $\zeta(0)$ and $\zeta(1)$. In particular, one of $\zeta(0)$ and $\zeta(1)$ is an initial (and end) segment of the other.
\end{thm}

\begin{proof} There is $k\geq 0$ such that $|\zeta^{2^k}(0)|\geq |v|$. Since $\zeta^{2^k}$ also generates $V$, i.e. $V=\lim_n \zeta^{n2^k}(0)$, Proposition~\ref{mainprop} gives that $v$ is a periodic building block of both $\zeta^{2^k}(0)$ and $\zeta^{2^k}(1)$. Thus $\zeta^{2^k}$ is Euclidean. By Lemma~\ref{lem8}, $\zeta$ is Euclidean. Thus there is $\alpha\in \{0,1\}^{<\N}$ that is a periodic building block of both $\zeta(0)$ and $\zeta(1)$. It follows that $\alpha$ is a periodic building block of $V$. Since $v$ is principal, $v$ is a periodic building block of $\alpha$. Hence $v$ is a periodic building block of both $\zeta(0)$ and $\zeta(1)$.
\end{proof}

\subsection{Coupled words and coupled substitutions}

We move now introduce a new possible criterion for substitution sequences which allows us to identify a large class of substitutions as having spacer rank greater than one.


\begin{defn}
For $a\in \{0,1\}$, denote $1-a$ by $a^c$. We call a word $u$ \textbf{coupled} if it is of the form
$$ u=a_1a_1^c\cdots a_ka_k^c $$
for $k\geq 1$ and $a_1,\dots, a_k\in\{0,1\}$. Likewise, an infinite word $V$ is \textbf{coupled} if all of its initial segments of even lengths are coupled. 
We call a substitution $\zeta:\{0,1\}\to \{0,1\}^{<\N}$ \textbf{coupled} if both $\zeta(0)$ and $\zeta(1)$ are coupled. 
An example of a coupled substitution is the \textbf{Morse substitution} $0\mapsto 01, 1\mapsto 10$, which generates the Morse sequence $V=\lim_n\zeta^n(0)$.
\end{defn}

The following lemma is easy to see.

\begin{lem} If a substitition $\zeta$ is coupled, then its limit $V=\lim_n \zeta^n(0)$ is coupled.
\end{lem}

\begin{thm} An aperiodic, coupled, infinite word $V$ does not have a rank one construction.
\end{thm}

\begin{proof} It is easy to see that $V$ does not contain $000$ or $111$. We assume that $V$ starts with $0$. By the coupledness, $V$ in fact starts with $01$.
Note that if $V$ does not contain $11$, then it is periodic with $01$ repeated indefinitely, contradicting our assumption that $V$ is aperiodic. Thus we must have that $V$ contains an occurrence of $11$. 
Now assume $V$ is built from a word $v$ of length $>1$ that starts and ends with $0$ and contains an occurrence of $11$. Let $h$ be the length of $v$. We write $v=a_1\cdots a_h$. We consider two cases.

Case 1. $h$ is even. Let $a_ia_{i+1}$ be the leftmost occurrence of $11$ in $v$. Note that $i$ must be even. Let $m$ be the largest positive integer such that we can write $V$ as
$$ V=v^m10\cdots. $$
Such $m$ exists since $V$ is aperiodic. Since $V$ is built from $v$, we can write $V=v^m1v\cdots$. In other words, the demonstrated $0$ in the above expression is the beginning of another copy of $v$. Now consider $a_ia_{i+1}$ in this copy of $v$, which is 11. Since $i$ is even, this $a_i=1$ occurs in $V$ at an odd position, contradicting the coupledness of $V$.

Case 2. $h$ is odd. Since $V$ is built from $v$, we can write
$$ V=v1^{t_1}v1^{t_2}\cdots.$$
Since $V$ does not contain $111$, each $t_i\in\{0,1,2\}$. Since $h$ is odd, $t_1\geq 1$. By the coupledness of $V$, and by an easy induction, we can see that if $t_1=\cdots=t_k=1$, then $t_{k+1}\geq 1$. If $t_i=1$ for all $i\geq 1$, then $V$ is periodic with the initial segment $v1$ repeated indefinitely, contradicting the assumption that $V$ is aperiodic. Thus for some $i\geq 1$, $t_i=2$. Now an argument similar to Case 1 gives a contradiction. In fact, let $i$ be the smallest with $t_{i}=2$ and consider the $(i+1)$-th copy of $v$ in $V$. Its leftmost occurrence of $11$ takes place at an odd position in $V$, contradicting the coupledness of $V$.
\end{proof}

\begin{cor}\label{C:notrankone} If $\zeta$ is a coupled substitution and $V=\lim_n \zeta^n(0)$ is aperiodic, then $V$ does not have a rank one construction.
\end{cor}

In particular, we achieve that the Morse sequence --  a fixed point of the substitution $0 \mapsto 01$ and $1 \mapsto 10$ -- is not a rank one word. Moreover, by Theorem \ref{periodic}, for a word $V$ that is a fixed point of a nontrivial substitution, the only possible periodic and coupled $V$ is $V = (01)^\infty$.


\subsection{Rank one words that are  fixed points of substitutions}
There are some special classes of substitutions that we can prove are rank one, as well as a few that we can prove are not rank one. A few useful definitions for these cases are the following.

\begin{defn}
A substitution $\zeta: \{0,1\} \to \{0,1\}$ is called a \textbf{proper} substitution if for all $a \in \{0,1\}$, $\zeta(a)$ begins with the same letter and ends with the same (potentially different from the first) letter.
\end{defn}

\begin{defn}
A substitution $\zeta: \{0,1\} \to \{0,1\}$ is called a \textbf{primitive} substitution if for all $a \in \{0,1\}$, every element of $\{0,1\}$ is in $\zeta (a)$. More generally, a substitution is called \textbf{eventually primitive} (\textbf{proper}) if there exists $ n \in \N$ such that $\zeta^n$ is primitive (proper).
\end{defn}

\begin{prop}
Suppose $\zeta$ is a not-eventually primitive substitution with $V = \lim_{n\to \infty} \zeta^n(0)$ as a fixed point. Then $V$ has a rank one construction.
\end{prop}

\begin{proof}
First, if $\zeta$ is a not eventually primitive substitution, then we can we can split into $4$ cases that can each be dealt with simply.
\begin{enumerate}
\item Suppose $\zeta(0) = 0^{n_1}$ and $\zeta(1) = v$. Then $V = 000 \ldots$ is the fixed point, which is periodic and therefore rank one.
\item Suppose $\zeta^k(0) = 1^{n_k}$ and $\zeta(1) = v$ for all $k$. If $v$ begins with $0$, the limit above does not converge. Now, if $v$ begins with $1$, then we have $\zeta^2(0) = 1^{n_2} = \zeta(1^{n_1}) = v^{n_1}$ implies that $v = 1^{n_2/n_1}$ so we have $V = 111 \ldots$ is our fixed point, which is although not technically a rank one word, it is periodic and defines a one-point dynamical system.

\item Suppose $\zeta(0) = v$ and $\zeta(1) = 1^{m_1}$. Then if $v$ starts with $1$, then $V = 111 \ldots$ is the fixed point which we discussed above. If $v$ starts with $0$, then we have that $v$ maps to something built by $v$, because every $0$ in $v$ gets replaced with $v$ and every $1$ gets replaced with $m_1$ spacers. So we have that $\zeta^n(0)$ is built by $\zeta^{n-1}(0)$, so we have an infinite number of words building $V$ so $V$ is symbolic rank one.

\item Suppose $\zeta(0) = v$ and $\zeta^k(1) = 0^{m_k}$ for all $k$. Then if $v$ begins with a $1$, then $V = 000 \ldots $ is our fixed point which like above is periodic and therefore rank one. Now, note that $\zeta(1) = 0^{m_1} \implies \zeta(\zeta(1)) = \zeta(0)^{m_1} = v^{m_1} = 0^{m_2}$, so we must have that $v = 0^{m_2/m_1}$ and so we still have $V = 000 \ldots$ as our fixed point which is rank one.
\end{enumerate}
In all cases, we have that $V$ is either periodic or has an infinite number of finite words building it.
\end{proof}

Note that in this proof, the first and third cases  used the weaker condition of not-primitive which does not hold for the the second and fourth case as can be seen by the Fibonacci substitution, which is primitive but not eventually primitive,  as the second power of the substitution is not primitive. However, even in the second and fourth case, we only needed that both $\zeta$ and $\zeta^2$ are not primitive. 

Theorem~\ref{T:rank2sub}  shows that certain sequences that are fiexed-points of substitutions cannot be rank one, and so therefore must be of spacer rank two. We only consider a special case of constant length proper substitutions with a certain beginning and end. For another possible approach to show that some substitutions are of spacer rank two we note that Gao and Ziegler in \cite{GaZi18} have shown that an infinite odometer cannot not be a factor of a rank one shift. 
Thus one way to show that a subshift is not rank one is to show that it has an infinite odometer factor.

\begin{thm}\label{T:rank2sub}
Suppose that  $\zeta$ is a proper, constant-length, substitution such that the word \\ $V = \lim_{n \to \infty} \zeta^n(0)$ is aperiodic and the first and last letters of $\zeta(a)$ are different. Then $V$ has a spacer rank two construction. 
\end{thm}

\begin{proof}
We will only show one case the case for proper substitutions of the form $\zeta(a) = 0\ldots 1$, but the case follows similarly for $\zeta(a) = 1 \ldots 0$. 
Now suppose $\zeta : \{0,1\} \to \{0,1\}^{<\N}$ is an aperiodic substitution such that $\zeta(0) = 0 \ldots 1$ and $\zeta(1) = 0 \ldots 1$ with $|\zeta(0)| = k = |\zeta(1) |$ and $n$ being the largest $n < k$ such that $1^n$ is a subword of $\zeta(0)$ or $\zeta(1)$. 
Note that $1^{n+1}$ cannot appear as a subword of $V = \lim_{n \to \infty} \zeta^n(0)$ and suppose for contradiction that $V$ is a rank-one word. Then there exists $ v$ beginning and ending with $0$ such that $v$ builds $V$ and $|v| > (k^2 - 1)n$, because there must be infinitely many words that build a rank-one word. 
Since $v$ builds $V$, then we will have that $\zeta(v)$ agrees with $V$ for the first $k|v|$ terms. 
That means that we can write \[\zeta(v) = v1^{a_1}v \ldots v 1^{a_{k-1}} w_S\] where $0 \leq a_i \leq n$ and $S = \sum_{i=1}^{k-1} a_i$ and $w_S = v \upharpoonright |v| - S$. 
Note that $S \leq (k-1)n$ so $w_S$ is a nonempty word. 
Then, note that we similarly have \[\zeta(w_S) = v1^{a_1}v \ldots v 1^{a_{k-1}} w_{T}\] where the $a_i$ are the same as in $\zeta(v)$ and $T = (k+1)S \leq (k^2 - 1)n$ so $w_T$ is nonempty as well. 
Now, we must have that $V \upharpoonright k^2|v| = \zeta^2(v)$ since $v$ builds $V$, so using the above expressions, we find that \[V = v1^{a_1}v \ldots v 1^{a_{k-1}} w_S \zeta(1)^{a_1} v1^{a_1}v \ldots v 1^{a_k} w_S \ldots v1^{a_1}v \ldots v 1^{a_{k-1}} w_S \zeta(1)^{a_{k-1}} v1^{a_1}v \ldots v 1^{a_{k-1}} w_T \ldots\]

Now, we can split into a few different cases and complete the proof by using the fact that $V$ as written above, must still be built by $v$. First, note that if $S = 0$, then we just have $V = vvv \ldots$ which is a contradiction of aperiodicity. 
Next, note that if the $S^\text{th}$ letter after the first copy of $w_S$ is a $1$, then that is a contradiction of the fact that $v$ ends with $0$. 
Now, note that if the $S^\text{th}$ letter after the first copy of $w_S$ is a $0$, then it must be followed by a $1$ or another copy of $v$. We will then break up this case into a few more cases to get all the contradictions we need.

If $S = ka_1$, then there is a contradiction because the $ka_1^\text{th}$ letter will be the last letter of $\zeta(1)^{a_1}$ which is set as $1$. 
If $S = ka_1 - 1$, and all the $a_i$ are equal, then $V$ is periodic or the $v$ does not match up with the letters in $\zeta(1)$, either of which is a contradiction. 
If there exists an $a_i > a_1$, then there will be extra $0$'s that cannot be accounted for in copies of $v$, so $v$ does not build $V$, a contradiction. 
If there exists $s a_i < a_1$, this will be functionally identical to the case of $S > ka_1$ which we will deal with last. Next we have if $S < ka_1 - 1$, then there are extra $0$'s that cannot be accounted for in copies of $v$ due to the size of $v$, so $v$ does not build $V$, a contradiction.

Finally, suppose $S > ka_1$. 
Note that we still have $S = \sum_{i=1}^{k-1} a_i \leq (k-1)n$ and $|v| > (k^2 - 1)n$. 
By comparison, we have that $w_S \zeta(1)^{a_1} \zeta(v) = v1^{b_1} v \cdots$ with the same restrictions of $b_i$ as are on the $a_i$ (i.e. they are positive integers $\leq n$). 
Now, note that $|w_S \zeta(1)^{a_1} \zeta(v)| = (k+1)|v|+ka_1 - S$. 
So we have that \[k|v|<|w_S \zeta(1)^{a_1} \zeta(v)| < (k+1)|v|\]
So we must have that there are at least $B = |v|+ka_1 - S$ ones due to spacers. However, $B > (k^2 - 1)n + k(n-1) - (k-1)n = k^2 - k$. Since the spacers are split into $k-1$ groups, we must have that some spacers come in a substring of length $\geq k$. However, $n < k$, so this is a contradiction because we cannot have more than $n$ $1$'s in a row.
\end{proof}  

Now we study condition for when a substitution determines a rank one word. We note that if $\zeta(0)$ contains only one $0$ and $\zeta(1)$ does not contain $0$, then $V=01^\infty$ and hence it is not a rank one word. We introduce the following definition.

\begin{defn}
We call a substitution $\zeta$ \textbf{adequate} if $\zeta(0)$ contains two $0$s and $\zeta(1)$ contains $0$.  
\end{defn}

If $\zeta$ is adequate, then $\zeta^n$ is adequate for all $n\geq 1$, and  $(|\zeta^n(0)|)_{n\geq 1}$, and $(|\zeta^n(1)|)_{n\geq 1}$ are both strictly increasing.

\begin{lem}\label{lemr1} Let $V$ be an infinite word built  by a nontrivial substitution $\zeta$. If $\zeta(0)$ contains two $0$s and $\zeta(1)$ does not contain $0$, then $V$ is a rank one word.
\end{lem}

\begin{proof} Assume $\zeta(0)=\alpha1^a$, where $\alpha$ starts and ends with $0$, and $a\geq 0$, and $\zeta(1)=1^b$, where $b\geq 1$. Then we claim that for all $k\geq 1$, $\zeta^k(0)$ is built from $\alpha$. We prove this by induction. For $k=1$ this is obvious. Next, let $\alpha=0c_1c_2\dots c_n$, where $c_1, \dots, c_n\in\{0,1\}$, then
$$ \zeta^{k+1}(0)=\zeta^k(\alpha1^a)=\zeta^k(0)\zeta^k(c_1)\zeta^k(c_2)\dots\zeta^k(c_n)\zeta^k(1)^a.$$
Since $\zeta^k(1)$ does not contain $0$ for any $k\geq 1$, from the inductive hypothesis that $\zeta^k(0)$ is built from $\alpha$, we get that $\zeta^{k+1}(0)$ is also built from $\alpha$. 

Now it follows from the claim that $V=\lim_{k\to\infty}\zeta^k(0)$ is also built from $\alpha$. Since $V$ can also be obtained from substitution $\zeta^2$, a similar argument gives that $V$ is built from a word that is longer than $\alpha$. Repeating this, we conclude that there are infinitely many finite words $\beta$ such that $V$ is built from $\beta$. This implies that $V$ is rank one.
\end{proof}

For the rest of this subsection, consider an adequate substitution $\zeta$ and $V=\lim_n\zeta^n(0)$.

\begin{lem}\label{lemr2} Let $\zeta$ be an adequate substitution, and suppose
$$\begin{array}{l} \zeta(0)=01^{s_1}\cdots 01^{s_k} \\ \zeta(1)=1^{t_0}01^{t_1}\cdots 01^{t_l}\end{array} $$
where $k\geq 2$, $s_i\geq 0$ for $i=1,\dots, k$, $l\geq 1$, $t_j\geq 0$ for $j=0,\dots, l$. Let $V=\lim_n \zeta^n(0)$. 
Then the length of a maximal block of $1$s in $V$ is one of the following numbers:
$$ s_1, \dots, s_k, t_1,\dots, t_l, s_k+t_0, t_l+t_0. $$
\end{lem}


\begin{proof}
Write $V=a_1a_2\cdots$ where $a_m\in\{0,1\}$ for $m\geq 1$. Then $V=\zeta(a_1)\zeta(a_2)\cdots$. A maximal block of $1$s in $V$ must occur between two $0$s in an occurrence of $\zeta(0)\zeta(0)$, $\zeta(0)\zeta(1)$, $\zeta(1)\zeta(0)$, or $\zeta(1)\zeta(1)$. By observation, the length of a maximal block of $1$s in $V$ must be one of the numbers listed.
\end{proof}

This means that if $V$ is an unbounded rank-one word, then it has bounded spacer parameter.

Recall from \cite[\S2.4]{GaHi16AT} that if $W$ is an unbounded rank one word, we can define $L_W(i)$ to be the length of the $i$-th maximal block of $1$s in $W$. If $W$ is aperiodic and is built from a finite word $w$ which starts and ends with $0$ and there are $r$ many $0$s in $w$, then $L_W$ is periodic on the congruence classes $i\not\equiv 0$ mod $r$ and aperiodic on the congruence class $i\equiv 0$ mod $r$. By Corollary 2.4 (b) of \cite{GaHi16AT}, if the spacer parameter of $W$ is bounded by $B$ and $B<|w|$, then there is $T_r$ such that, if an occurrence of $w$ in $W$ is preceded by $p$ many maximal blocks of $1$s in $W$, then this occurrence of $w$ is expected if and only if  $L_W(p+r), L_W(p+2r),\dots, L_W(p+T_rr)$ are not all equal.

For the above fixed adequate substitution $\zeta$ and the infinite word $V$ that is a fixed point of $\zeta$, we let $L$ denote the function $L_V$. 

\begin{prop}\label{propr} Let $\zeta$ be an adequate substitution and $V=\lim_n\zeta^n(0)$. Assume $V$ is an aperiodic rank one word and $V$ is built from a finite word $v$. Then for sufficiently large $n$, both $\zeta^n(0)$ and $\zeta^n(1)$ are built from $v$.
\end{prop}

\begin{proof} By Lemma~\ref{lemr2} there is an upper bound $B$ for the spacer parameter of $V$. Without loss of generality assume $|v|>B$. If this does not hold, we can consider a finite word $v'$ such that $|v'|>B$, $v'$ is built from $v$, and $V$ is built from $v'$. If the conclusion of the proposition holds for $v'$, then it holds for $v$.

Let $r$ be the number of $0$s in $v$. Let $T_r$ be the number given by Proposition~2.4(b) of \cite{GaHi16AT} mentioned above. That is, if an occurrence of $v$ in $V$ is preceded by $p$ many maximal blocks of $1$s, then this occurrence of $v$ is expected iff $L(p+r), L(p+2r), \dots, L(p+T_rr)$ are not all equal. 

Let $n$ be sufficiently large such that, if we denote $\zeta^n$ by $\tau$, and write
$$\begin{array}{l} \tau(0)=\zeta^n(0)=01^{s_1}\cdots 01^{s_k} \\ \tau(1)=\zeta^n(1)=1^{t_0}01^{t_1}\cdots 01^{t_l}\end{array} $$
where $k\geq 2$, $s_i\geq 0$ for $i=1,\dots, k$, $l\geq 1$, $t_j\geq 0$ for $j=0,\dots, l$, then
\begin{enumerate}
\item[\rm (a)] $k, l\geq T_rr+1$;
\item[\rm (b)] there are $i_0<i_1<k$, $0<j_0<j_1<l$ with $r\mid (i_1-i_0), (j_1-j_0)$ and $s_{i_0}\neq s_{i_1}, t_{j_0}\neq t_{j_1}$. 
\end{enumerate}
Note that $\tau$ is an adequate substitution generating $V$. Since $\tau(0)$ is an initial segment of $V$, we must have $i_0\equiv i_1\equiv 0$ mod $r$. Call the $i_0$-th and $i_1$-th maximal blocks of $1$s in $\zeta(0)$ and the $j_0$-th and $j_1$-th maximal blocks of $1$s in $\zeta(1)$ {\em special blocks}.

Assume first that $V$ contains $00$. Then $V$ also contains $\tau(0)\tau(0)$. Consider an arbitrary occurrence of $\tau(0)\tau(0)$ in $V$. Suppose there are $p$ many maximal blocks of $1$s in $V$ before this occurrence of $\tau(0)\tau(0)$. Then the special blocks in the occurrence of $\tau(0)\tau(0)$ have the following indices in $L$:
$$ p+i_0, p+i_1, p+k+i_0, p+k+i_1. $$
Since $L(p+i_0)=s_{i_0}\neq s_{i_1}=L(p+i_1)$ and $L(p+k+i_0)=s_{i_0}\neq s_{i_1}= L(p+k+i_1)$, these indices must be in the same congruence class mod $r$. In particular, $k=(p+k+i_0)-(p+i_0)\equiv 0$ mod $r$. Thus $k$ is a multiple of $r$, and $\tau(0)$ is built from $v$. Write
$$ \tau(0)=v1^{u_1}\cdots v1^{u_{k/r}}. $$
The demonstrated occurrences of $v$ in this expression of $\tau(0)$ are called {\em $\tau(0)$-expected occurrences}.

Since $\tau(0)$ is an intial segment of $V$, its first occurrence of $v$ in the first occurrence of $\tau(0)$ in $V$ is expected, and therefore $L(r)=s_r, L(2r)=s_{2r}, \dots, L(T_rr)=s_{T_rr}$ are not all equal. Also because $k\geq T_rr+1$, the maximal blocks of $1$s corresponding to the indices $r, 2r, \dots, T_rr$ all appear in $\tau(0)$. Now consider any occurrence of $\tau(0)$ in $V$. Assume that there are $p$ many maximal blocks of $1$s in $V$ preceding this occurrence of $\tau(0)$. Then $L(p+r)=L(r), L(p+2r)=L(2r), \dots, L(p+T_rr)=L(T_rr)$ are not all equal, and therefore the first occurrence of $v$ in this occurrence of $\tau(0)$ must be expected in $V$. It follows that all the $\tau(0)$-expected occurrences of $v$ in this occurrence of $\tau(0)$ in $V$ must be expected in $V$. 

Now we must have that $V$ contains $01$. Consider an arbitrary occurrence of $\tau(0)\tau(1)$ in $V$. Suppose there are $p$ many maximal blocks of $1$s in $V$ before this occurrence of $\tau(0)\tau(1)$. Since all $\tau(0)$-expected occurrences of $v$ in this occurrence of $\tau(0)$ are expected in $V$, it follows that the occurrence of $01^{t_1}\cdots 1^{t_{r-1}}0$ in this occurrence of $\tau(1)$ must be an expected occurrence of $v$ in $V$. Moreover, $L(p+k/r+r)=t_{r}, L(p+k/r+2r)=t_{2r}, \dots, L(p+k/r+T_rr)=t_{T_rr}$ are not all equal. Again, since $l\geq T_rr+1$, all these maximal blocks of $1$s appear in $\tau(1)$. It follows that, in any occurrence of $\tau(1)$ in $V$, the occurrence of $01^{t_1}\cdots 1^{t_{r-1}}0$ in this occurrence of $\tau(1)$ must be an expected occurrence of $v$ in $V$. 
As a consequence, $j_0\equiv j_1\equiv 0$ mod $r$.

We next claim that $l$ must be a multiple of $r$. To see this, note that $V$ must contain $10$ and therefore an occurrence of $\tau(1)\tau(0)$. We consider a particular occurrence of $\tau(1)\tau(0)$ in $V$. All $\tau(0)$-expected occurrences of $v$ in this occurrence of $\tau(0)$ are expected in $V$, and the occurrence of $01^{t_1}\cdots 1^{t_{r-1}}0$ in this occurrence of $v$ are expected in $V$. It follows that the word in between these expected occurrences of $v$ is built from $v$. Thus $\tau(1)$ is built from $v$, and $l$ is a multiple of $r$.

Thus we have completed the proof of the proposition under the condition that $V$ contains $00$. 

Next we assume that $V$ contains $11$ but not $00$. Consider an arbitrary occurrence of $\tau(1)\tau(1)$ in $V$. Suppose there are $p$ many maximal blocks of $1$s in $V$ before this occurrence of $\tau(1)\tau(1)$. Then the special blocks in the occurrence of $\tau(1)\tau(1)$ have the following indices in $L$:
$$ p+j_0, p+j_1, p+l+j_0, p+l+j_1. $$
Since $L(p+j_0)=t_{j_0}\neq t_{j_1}=L(p+j_1)$ and $L(p+l+j_0)=t_{j_0}\neq t_{j_1}= L(p+l+j_1)$, these indices must be in the same congruence class mod $r$. In particular, $l=(p+l+j_0)-(p+j_0)\equiv 0$ mod $r$. Thus $l$ is a multiple of $r$. 

Similar to the previous case, we note that the first occurrence of $v$ in $\tau(0)$ occurs expectedly in any occurrence of $\tau(0)$ in $V$ because $\tau(0)$ occurs in $V$ as an initial segment and $s\geq T_rr+1$. Note that there must be an occurrence of $01^q0$ in $V$ for some $q\geq 1$. Consider a particular occurrence of $\tau(0)\tau(1)^q\tau(0)$ in $V$. Between the first occurrences of $v$ in the two occurrences of $\tau(0)$ there are $(s-r)+ql$ many maximal blocks of $1$s. Since these occurrences of $v$ are both expected, the word in between them is built from $v$, and thus $(s-r)+ql$ is a multiple of $r$. It follows that $s$ is a multiple of $r$, and $\tau(0)$ is built from $v$. Moreover, the first occurrence of $01^{t_1}\cdots 1^{t_{r-1}}0$ in the first occurrence of $\tau(1)$ in this occurrence of $\tau(0)\tau(1)^q\tau(0)$ is an expected occurrence of $v$. Since $l$ is a multiple of $r$, it follows that $\tau(1)$ is also built from $v$. 

This completes the proof of the proposition, since if $V$ does not contain either $00$ or $11$, then $V=(01)^\infty$ is periodic.
\end{proof}

\begin{lem}\label{lemr3} Let $\zeta$ be an adequate substitution and $V=\lim_n\zeta^n(0)$ be an aperiodic rank one word. Suppose
$\zeta(1)$ starts with $1$. Write $\zeta(1)=1^tu$ where $t\geq 1$ and $u$ starts with $0$. Then one of $\zeta(0)$ and $u$ is an initial segment of the other. 
\end{lem}
\begin{proof}
Since $\zeta(1)$ starts with $1$, $\zeta(1)$ is an initial segment of $\zeta^n(1)$ for all $n\geq 1$. Since $V$ is an aperiodic rank one word, there is a finite word $v$ with $|v|>|\zeta(0)|, |u|$ such that $V$ is built from $v$. By Proposition~\ref{propr}, there is $n\geq 1$ such that both $\zeta^n(0)$ and $\zeta^n(1)$ are built from $v$. Since $\zeta(0)$ is an intial segment of $\zeta^n(0)$, we have that $\zeta(0)$ is an intial segment of $v$. Likewise $\zeta(1)=1^tu$ is an initial segment of $\zeta^n(1)$, and therefore $u$ is also an intial segment of $v$. Thus one of $\zeta(0)$ and $u$ is an initial segment of the other. 
\end{proof}

This lemma can be used to show some words are not rank one. In particular, we obtain another proof that the Morse word is not a rank one word: consider $\zeta^2(0)=0110$ and $\zeta^2(1)=1001$; since $V=\lim_n\zeta^{2n}(0)$ is aperiodic (by Theorem~\ref{periodic}),  and neither $0110$ nor $001$ is an initial segment of the other, it follows from Lemma~\ref{lemr3} that $V$ is not a rank one word.

\subsection{Generalized Morse sequences}
We move forward to analyze the spacer rank of sequences which are more general than those that are fixed points of substitutions. We now recall a definition of generalized Morse sequences, which were  defined in \cite{Ke68}. The Morse sequence is an example of a  generalized Morse sequence.

\begin{defn}
A \textbf{generalized Morse sequence} is an element $w$ of $\{0, 1\}^\mathbb{N}$ that can be written as \[w = b_0 \times b_1 \times b_2 \times \cdots \]
where each $b_i$ is a finite word beginning with $0$, which we call a block.
\end{defn}

Note that when we say $a \times b$ we mean to take copies of $a$ and $a^c$ and concatenate them such that whenever there is a $0$ in $b$, we place $a$, and whenever there is a $1$ in $b$, we place $a^c$. So we get that $01 \times 010 = 011001$.

\begin{lem}[\cite{Ke68}, Lemma 1]
A generalized Morse sequence $w \in \{0,1\}^\N$ with $w = b_0 \times b_1 \times b_2 \times \cdots$ is periodic if and only if  there exists $k$ in $\N$ such that $b_k \times b_{k+1} \times b_{k+2} \times \cdots = 000000\ldots$ or $b_k \times b_{k+1} \times b_{k+2} \times \cdots = 010101\ldots$.
\end{lem}

\begin{prop}
All generalized Morse sequences have an at most spacer rank two construction. If the sequence of blocks is periodic, then the generalized Morse sequence is a fixed point of  a  substitution.
\end{prop}

\begin{proof}
Let $w = b_0 \times b_1 \times b_2 \times \cdots$ be a generalized Morse  sequence. Then consider $v_{1,k} = b_0 \times b_1 \times \cdots \times b_k$ and $v_{2,k} = v_{1,k}^c$. Then we have that $\{v_{1,k}, v_{2,k}\}$ builds $w$ for all $k$ without using any spacers, but this has the issue that the words might begin and end with $1$. This can be fixed by removing the leading and ending $1$'s from the $v_{i,k}$ and adding them back as spacers. So if $v_{1,k} = 1^a w_{1,k} 1^b$ and $v_{2,k} = 1^c w_{2,k} 1^d$ where $a,b,c,d \in \N$ are as large as possible (possibly $0$), then we have that $w$ is built by $\{w_{1,k},w_{2,k}\}$ for all $k$, where instead of just concatenating, whenever we previously has $v_{1,k}v_{2,k}$ we instead have $w_{1,k}1^{b+c}w_{2,k}$ and similarly for the other combinations.
\end{proof}

We note that if a sequence of blocks is eventually periodic, then the spacer rank two construction is simpler to state than in the general case.

\begin{example} \rm{
Consider the generalized Morse  sequence $w = 010 \times 01 \times 01 \times \cdots$. Then let $w_{1,1} = 010$ and $w_{2,1} = 0$, then we can define the next set of words that builds $w$ by the following. If $n$ is odd, 
\[w_{n+1,1}=w_{n,1}1w_{n,2} \text{ and } w_{n+1,2} = w_{n,2}1w_{n,1}.\]
If $n$ is even, then we instead have the following rule,
\[w_{n+1,1}=w_{n,1}11w_{n,2} \text{ and } w_{n+1,2} = w_{n,2}w_{n,1}.\]
This example follows the same rule as Morse, though this depends both on the repeated word and the initial word, so finding other rules would not be difficult.
}\end{example}

\section{Finite Spacer Rank Subshifts}\label{S:finiterank}

In this section we study finite spacer rank subshift systems. First we recall some basic definitions. The space $\{0,1\}^\Z$ is given the product topology where each $\{0,1\}$ has the discrete topology; this topology is metrizable and  $\{0,1\}^\Z$ is a Cantor space. The \textbf{shift map} $\sigma: \{0,1\}^\Z\to\{0,1\}^\Z$ is defined by 
\[\sigma(x){(i)} = x{(i+1)}\text{  
for all }\ i \in \mathbb{Z}.\] This shift is a homeomorphism of $\{0,1\}^\Z$.

\begin{defn} \label{rank one System}
Let $V \in \{0, 1\}^\mathbb{N}$ be an infinite word. We define the system $X_V$, a \textbf{word subshift}, by
$$X_V = \{x \in \{0, 1\}^\mathbb{Z} : \text{ every finite subword of } x \text{ is a subword of } V \}.$$
It is clear that $X_V \subset{\{0, 1\}^\mathbb{Z}}$ is closed, hence  compact,  and invariant under the shift $\sigma$, i.e., $\sigma(X_V)=X_V$. One can  verify that $\sigma$ is a homeomorphism of $X_V$. We say that $(X_V, \sigma)$ is the \textbf{subshift system}, or \textbf{word subshift system}, associated to $V$. We define the \textbf{orbit} of $x \in X$ to be the set $\{x: \sigma^n(x),: n\in\Z\}$.
\end{defn}

We note that a \textit{shift space}, see \cite{LiMa95}, (also called in the literature a \textit{subshift}, see e.g. \cite{Pa22}) is a closed (hence compact) shift invariant subset of $\{0,1\}^\Z$ (or more generally of $\{0,1,\ldots, n-1\}$). It is clear that word subshifts are shift spaces, but Example~\ref{E:notsubshift} shows that there are shift spaces that are not word subshifts.

\begin{example} \rm{
Consider the word
\[V=1011010011001000110001\cdots\]
(The number of zeros grows and the number of 1s alternate between 1 and 2.)
We note that the system $X_V$ consists of the following three kinds of words:
(1) constant 0; (2) bi-infinite words with exactly one 1; (3) bi-infinite words with exactly two consecutive $1$s.
Therefore there is no element of $X_V$ that generates $X_V$.
}\end{example}

\begin{example} \rm{\label{E:notsubshift}
Consider the following bi-infinite word
\[z(n)=1 \text{ for }n<0,\  \text{and }z(n)=0 \text{ for }n\geq 0,\]
and let $Z$ be the closure of the orbit of $z$. 
The system $Z$ consists of the orbit of $z$ plus two more elements: the constant sequence 0 and the constant sequence  1.
We claim that $Z$  is not generated by any infinite word $V$, i.e., $Z$ is not of the form $X_V$ for some $V$. We proceed by contradiction and assume that $Z=X_V$ for some infinite word $V$. Then $V$ contains all finite words of the form $1^n0^m$ for $n, m$ nonnegative integers. So there are infinitely many occurrences of $01$ in $V$. It follows that $Z$ must contain an element in which 01 occurs, a contradiction.
}\end{example}




One can verify that if $V \in \{0,1\}^\mathbb{N}$ is an infinite word, then the system $(X_V, \sigma)$   associated to $V$   is nonempty, and if $V$ is recurrent, then $X_V$ is finite or a Cantor set. Also,  for any finite subword and any position, there has to be an element in $X$ that contains the subword at that specific position.

\begin{defn} \label{rank n systems}
Let $n\geq 2$.  A subshift $X$ is a \textbf{spacer rank}-$\mathbf{n}$\textbf{ system} if there is a word $V$ with a spacer rank-$n$ construction such that $X=X_V$ and there is no word $W$ with a  spacer rank-$(n-1)$ construction such that $X=X_W$.
The system is a \textbf{(spacer) rank-one system} if there is a (spacer) rank-one word $V$ such that $X=X_V$.

A word $V \in \{0, 1\}^\mathbb{N}$ is said to have \textbf{system spacer rank} $\mathbf{n}$ if $X_V$ is a spacer rank-$n$ system.
\end{defn}


There are many infinite words that are associated to the same system (see e.g. Example~\ref{e:tworanks}). Two different infinite words of different spacer ranks can be associated to the same symbolic shift system. In fact,  any symbolic shift system $(X, \sigma)$ is associated to infinitely many words.

\begin{example}\label{e:tworanks}  \rm{
We can obtain infinite words that are not  rank-one by adding a $0$ to the beginning of the Chac\'on sequence or by removing the first $0$ of the Chac\'on sequence.
Both words are clearly not rank one. However, the system associated to both words is rank one since it is the same system that is associated to the Chac\'on sequence.
}\end{example}

The example above shows a  rank-one system that is trivially associated to a word that is not  rank-one.
Given a rank-one word, the system associated to it is always rank one, though given a word with a spacer rank-$n$ construction, it is not obvious whether its associated system is of spacer rank $n$.
The proofs of the following lemma and theorem follow from the definitions and are left to the reader.

\begin{lem} \label{X_V subset X_W lemma}
For any infinite words $V$ and $W$, let $(X_V, \sigma)$ be the system associated to $V$ and $(X_W, \sigma)$ be the system associated to $W$. If every finite subword of $V$ is a subword of $W$, then $X_V \subset X_W$.
\end{lem}

\begin{thm} \label{X_V subset X_W prop}
For any infinite words $V$ and $W$, let $(X_V, \sigma)$ be the system associated to $V$ and $(X_W, \sigma)$ be the system associated to $W$. $X_V \subset X_W$ if and only if every finite subword $v$ that appears infinitely often in $V$ appears infinitely often in $W$.
\end{thm}

\begin{cor}
Let  $V$ and  $W$ be two  infinite words.  $W$. If $V$ and $W$  differ in finitely many digits, then $X_V = X_W$ 
\end{cor}
 


\begin{cor} \label{shift infinite word}
Let  $V$ and  $W$ be two  infinite words. If there exist $i, j \in \mathbb{N}$ such that \[V{(i)}V{(i+1)}V{(i+2)}\cdots = W{(j)}W{(j+1)}W{(j+2)}\cdots,\] then $X_V = X_W$. \end{cor}


\begin{cor}\label{cor:equivalent-systems}
Let  $V$ and  $W$ be two  infinite words. Then $X_V = X_W$ if and only if $V$ and $W$ have the same set of finite words that appear infinitely often in them.
\end{cor}

\begin{example} \rm{ While Proposition~\ref{P:allfiniterank} shows the existence of words with a proper spacer rank-$n$ construction for each $n$, here we give a natural example of a word with proper spacer rank-three construction.

Let
$$w_{0,1}=w_{0,2}=w_{0,3}=0,$$
$$w_{1,1}=01^30,\ w_{1,2}=01^40,\ w_{1,3}=01^50,$$
and for $n\ge1$,
$$\begin{array}{l} w_{n+1,1}=w_{n,1}(w_{n,2})^3w_{n,3}, \\
w_{n+1,2}=w_{n,2}(w_{n,3})^4w_{n,1},\\
w_{n+1,3}=w_{n,3}(w_{n,1})^5w_{n,2}.
 \end{array}
 $$
It is clear that this is a proper spacer rank-three construction for $W=\lim_{i\rightarrow \infty}w_{i,1}$. We claim that there is no lower spacer rank construction. We sketch a proof below. First note the following basic properties of $W$:
\begin{enumerate}
\item[(a)] $W$ starts with $01$, and all other occurrences of $0$s in $W$ are in blocks of size $2$;
\item[(b)] All occurrences of $1$s in $W$ are in blocks of size $3$, $4$ or $5$, and $W$ is uniquely readable as a word built from $w_{1,1}$, $w_{1,2}$ and $w_{1,3}$;
\item[(c)] If $x$ and $y$ are finite words that begin and end with $0$ and $|x|, |y|>30$, then there is at most one $a\in\{3,4,5\}$ such that $x1^ay$ is a subword of $W$.
\end{enumerate}
Also, $W$ is recurrent from the proper spacer rank-three construction given above, thus $X_W$ is a perfect set, and in particular $W$ is not eventually periodic. To see that $W$ is rank-three, we verify that for any words $u,v$ beginning and ending with $0$ and $|u|,|v|>30$, $W$ is not built from $u,v$. Toward a contradiction, assume
$$ W=w_01^{a_0}w_11^{a_1}\cdots $$
where $w_i\in\{u,v\}$ and $a_i\geq 0$ for all $i\in\mathbb{N}$.
Without loss of generality, assume $w_0=u$. Then $u$ begins with $01$. By (a) we have cases where $u$ ends with either $00$ or $10$, $v$ begins with either $01$ or $00$, and $v$ ends with either $00$ or $10$. By checking the cases, we conclude that we must have that both $u$ and $v$ begin with $01$ and end with $10$. For instance, consider the case $u$ ends with $00$ and $v$ ends with $10$. In this case it follows from (a) that for any $i\geq 1$, if $w_i=v$ then $w_{i+1}=u$ and $a_i=0$. Also by (c) there is a unique $a\in\{3,4,5\}$ such that $u1^av$ can be a subword of $W$. Thus
$$ W=u1^avu1^av\cdots $$
which is periodic, a contradiction.

Thus each of $u$ and $v$ is built from $w_{1,1}$, $w_{1,2}$ and $w_{1,3}$.

Define a substitution scheme $\zeta:\{0,1,2\}\to\{0,1,2\}^{<\mathbb{N}}$ by $$\zeta(0)=01^32,\ \zeta(1)=12^40,\ \zeta(2)=20^51. $$
Let $V=\lim_{n\rightarrow \infty}\zeta^n(0)$. Then $V$ represents the building of $W$ with the correspondence $0\mapsto w_{1,1}$, $1\mapsto w_{1,2}$ and $2\mapsto w_{1,3}$.
By the above discussion, we obtain two finite words $p,q\in \{0,1,2\}^{<\mathbb{N}}$ (corresponding to $u, v$ respectively) such that $V$ is built from $p, q$ without spacers, that is, 
$$ V=r_0r_1\cdots $$
where $r_i\in\{p,q\}$ for all $i\geq 0$. 

Let $p_0,q_0\in\{0,1,2\}^{<\mathbb{N}}$ be two finite words with the least value of $|p_0|+|q_0|$ such that $V$ is built from $p_0, q_0$ without spacers. It can be argued that either $p_0$ cannot be written as the form $\zeta(r)$ for some word $r\in\{0,1,2\}^{<\mathbb{N}}$ or $q_0$ cannot be written as the form $\zeta(r)$ for some word $r$. Finally, by straightforward but tedious calculations, we can see that $V$ is eventually periodic, which implies that $W$ is eventually periodic, a contradiction.

}\end{example}

We conclude this section with a computation of the topological entropy for some examples. In Theorem~\ref{P:tefr} we show that all shifts defined by finite spacer rank words have topological entropy zero, and in Example~\ref{E:polyunranked} we construct a word that does not have a spacer rank construction whose corresponding system cannot be defined by any spacer ranked word, and show that the system has topological entropy zero. We have seen that $P_n(V)$ is the complexity function of an infinite word $V$. If $X$ is a subshift its complexity $P_n(X)$ is defined to be the number of words of length $n$ in any $x\in X$ (see e.g., \cite{Kr03}). One can verify that $(\log P_n(X))_n$ is a subadditive sequence, so the limit $\lim_{n\to\infty} \frac{\ln P_n(X)}{n}$ exists and is defined to be the topological entropy of $X$. One can verify that in our case $P_n(X_V)=P_n(V)$, so the topological entropy of $(X_V,\sigma)$ is $\lim_{n\to\infty} \frac{\ln P_n(X)}{n}$.

\begin{thm}\label{P:tefr} 
If $V$ is an infinite word with a finite spacer rank construction, then the topological entropy of $(X_V,\sigma)$ is 0.
\end{thm}

\begin{proof}

If $V$ is a spacer rank-$m$ word, fix a spacer rank-$m$ construction $v_{i,j}, i\in\N$, $1\leq j\leq m$. For
$\varepsilon > 0$, there exist $k\in\N$, such that  $ \frac{\ln(m+k)}{
k}< \varepsilon$. Let  $i_0\in\N$ be such that
$\min_{1\leq j\leq m} |v_{i_0,j} | > k$, and define $k' = \max_{1\leq j\leq m} |v_{i_0,j}| $.  

There exists $N\in\N$
such that for every $n > N$
\[n + k_1 + k < \frac32 n,\text{ and } \frac{\ln(k_1+1)}{n}<\varepsilon.\]
 Then we define
$\zeta : \{0, 1, 2,\cdots ,m + k\} \to \{0,1\}^{<\N}$, by

\begin{equation}
\zeta(i) = 
\begin{cases}
1^i &  \text{ if } 0\leq i\leq k,\\
 v_{i_0,i-k}  &  \text{ if } k+1\leq i \leq m+k.
\end{cases}
\end{equation}

 Let $n_1 = 2([\frac{n+k_1}{k}] + 1)$. For $s\in \{0, 1, 2,\cdots ,m + k\}^{n_1}$ such that
$s(a)\geq k$ or $s(a + 1)\geq k$ for any $0\leq a\leq n_1-2, 0\leq b\leq k_1$, define $\phi(b, s)\in \{0,1\}^n$
by  \[\phi(b, s)(i) =\zeta(s)(i+b),\text{ for } 0\leq i\leq n-1.\] 

We can see that  for any $x\in X_V$,
and any subword $v$ of $x$, if $|v| = n$, then there exist $s\in  \{0, 1, 2,\cdots ,m+k\}^{n_1}$ and $0\leq b\leq k_1$ 
such that $\phi(b, s) = v$. Finally,
\begin{align*}
\frac{\ln((k_1+1)(m+k)^{2([\frac{n+k_1}{k}+1)}} {n}
&\leq \frac{\ln(k_1+1)}{n} + 3\frac{\ln(m+k)}{k}\\
&\leq 4\varepsilon.
\end{align*}
Thus the topological entropy of $X_V$ is 0.
\end{proof}

\begin{example} \rm{ \label{E:polyunranked} We construct a subshift of zero topological entropy that is not defined by any word with a spacer rank construction.   
Let
\begin{align*}
v_k&=01^k01^k0\text{ and }u_k=v_1v_2\cdots v_k, \text{ for }k\geq 1, \text{ and let }\\
V&= u_1u_2\cdots u_k\cdots \\
&= v_1v_1v_2\cdots v_1v_2\cdots v_k\cdots
 \end{align*}
We first show that the word $V$ has polynomial complexity, hence zero topological entropy. We   compute the complexity function $P_n(V)$. Let $S_n$ be the set of all subwords of $V$ of length $n$. Consider the finite word
$$ w=u_1u_2\cdots u_n=v_1v_1v_2\cdots v_1v_2\cdots v_n. $$
Then $S_n$ is exactly the set of all subwords of $w$ of length $n$. Since $|w|$ is a polynomial in $n$, we have that $P_n(V)=|S_n|$ is a polynomial in $n$. This implies that the topological entropy of $X_V$ is $0$.

Now let $S$ be the set of all finite words $w$ which occurs infinitely many times in $V$. Then $v_k\in S$ for all $k\geq 1$. It is easy to see that for any $w\in \{0,1\}^{<\mathbb{N}}$, $w\in S$ if and only if there is some $x\in X_V$ such that $w$ is a subword of $x$. Now suppose $W$ is an infinite word with $X_V=X_W$. We show that for all $k\geq 1$, $v_k$ is a subword of $W$. Then by Lemma 2.7, $W$ does not have a spacer rank construction. Fix $k\geq 1$. Since $v_k\in S$, we have that there is some $x\in X_V=X_W$ such that $w$ is a subword of $x$. Now $x\in X_W$ and $w$ is a subword of $x$, we must have that $w$ is a subword of $W$.
This proves that $X_V$ is not spacer ranked as a system.

}\end{example}

\begin{remark} \rm{If $C$ is the Chac\'on word, the word  $1C$ does not have a spacer rank construction for the trivial reason that it starts with $1$.
Note also that as we have previously seen, the word $0C$ has a spacer rank-two construction but is not (spacer) rank one.
However, we have that $X_C=X_{1C}=X_{0C}$, so they all define subshifts of rank one.
At the same time we have shown that if $V$ is a full complexity word and $U$ is any word such that $X_V=X_U$, then $U$ does not have a spacer rank construction.
}\end{remark}

\section{Spacer Rank for Sturmian Sequences}\label{S:sturmian}

We now consider a class of sequences called Sturmian sequences––of which the Fibonacci sequence is an example––and discuss the spacer rank of such sequences.
In Corollary~\ref{C:sturmr2} we prove that all Sturmian sequences define spacer rank-two systems.

\begin{defn}
A \textbf{Sturmian sequence} \cite{MH40} is an element $V$ of $\{0, 1\}^\mathbb{N}$ such that the number of words of length $n$ that are  subwords of $V$ is $n + 1$.
\end{defn}
We refer the reader to \cite{Fo02} for properties of the Sturmian sequence. The main properties of Sturmian sequences we use are the following:
\begin{itemize}
\item Sturmian sequences are not eventually periodic.
\item Sturmian sequences are \textbf{balanced}, meaning that for all $n$ and for any two subwords of length $n$, the nonnegative difference between the number of $1$'s between those two subwords of length $n$ is at most $1$.
\item In any Sturmian sequence, either $11$ appears as a subword or $00$ appears as a subword, but not both. If $11$ appears, we say that the sequence is $\textbf{type 1}$ and if $00$ appears, we say that the sequence is $\textbf{type 0}$.
\item If $V\in\{0,1\}^\mathbb{N}$ is a Sturmian sequence, then $V$ is recurrent.
\end{itemize}

We note that in  \cite{Fo02}, Sturmian sequences are characterized by the properties of being balanced and eventually periodic, but in the sequel we do not need this result.
 Sturmian systems are known to be minimal and uniquely ergodic \cite[Proposition 3.2.10]{Th20}, and with their unique invariant measure they are measurably conjugate to an irrational rotation \cite[Corollary 3.2.13]{Th20}, which we know is measurably rank one.

In the following we prove that any recurrent, balanced, non-eventually periodic infinite word starting with $0$ has a proper spacer rank-two construction. We also show that any Sturmian sequence (regardless of whether it begins with $0$ or with $1$) generates a spacer rank-$n$ subshift.

\begin{lem}\label{lem:VW} Let $V\in\{0,1\}^\mathbb{N}$ be a recurrent, balanced, non-eventually periodic word. If $00$ is a subword of $V$, then there exist a recurrent, balanced, non-eventually periodic word $W\in\{0,1\}^\mathbb{N}$ such that $00$ is a subword of $W$, $a_0,a_1\in\mathbb{N}^+$ such that $|a_0-a_1|=1$, and $0\leq b\leq \max\{a_0, a_1\}$ such that $$V=0^b10^{a_{W{(0)}}}10^{a_{W{(1)}}}10^{a_{W{(2)}}}1\cdots.$$
If $11$ is a subword of $V$, then there exist a recurrent, balanced, non-eventually periodic word $W\in\{0,1\}^\mathbb{N}$ such that $00$ is a subword of $W$, $a_0,a_1\in\mathbb{N}^+$ such that $|a_0-a_1|=1$, and $0\leq b\leq \max\{a_0, a_1\}$ such that $$V=1^b01^{a_{W{(0)}}}01^{a_{W{(1)}}}01^{a_{W{(2)}}}0\cdots.$$
\end{lem}

\begin{proof} First suppose $00$ is a subword of $V$. 
There exists a unique $b\in\mathbb{N}$ such that $0^b1$ is an initial segment of $V$. Since $V$ is balanced, there exists a unique $n\in\mathbb{N}^+$ such that between any two occurrences of $1$s in $V$ there can occur either $n$ or $n + 1$ many consecutive $0$s. Fix such an $n\in \mathbb{N}^+$. So $V$ can be written uniquely as
$$0^b10^{c_0}10^{c_1}1\cdots$$
where $c_i=n$ or $c_i=n+1$ for all $i\in\mathbb{N}$. Since $V$ is not eventually periodic, there exists $i\in\mathbb{N}$ such that $c_i=c_{i+1}$. Define $a_0=c_i$ where $i$ is the least such that $c_i=c_{i+1}$. Define $a_1=n+1$ if $a_0=n$, and $a_1=n$ if $a_0=n+1$. Then we can define $W$ in an obvious way. By definition, $00$ is a subword of $W$. Since $V$ is balanced, we have $0\leq b\leq n+1$.

We verify that $W$ is recurrent, balanced, and not eventually periodic. Since $V$ is recurrent and not eventually periodic, it is clear that $W$ is recurrent and not eventually periodic. Assume $W$ is not balanced. Let $m$ be the least integer such that there exist two subwords $v,w$ of $W$ of length $m$ so that the nonnegative difference between the numbers of $1$s in $v$ and in $w$ is at least $2$. Then we must have $m>1$, and by the minimality of $m$ the nonnegative difference between the numbers of $1$s in $v$ and in $w$ is exactly $2$. Define
$$\begin{array}{l}v_1=0^{a_{v(0)}}10^{a_{v(1)}}\cdots10^{a_{v(m-1)}}, \\ v_2=10^{a_{w(0)}}10^{a_{w(1)}}\cdots10^{a_{w(m-1)}}1, \\ v_3=0^{a_{w(0)}}10^{a_{w(1)}}\cdots10^{a_{w(m-1)}}, \\ v_4=10^{a_{v(0)}}10^{a_{v(1)}}\cdots10^{a_{v(m-1)}}1.
\end{array}
$$
Observe that the nonnegative difference between the numbers of $0$s in $v_1$ and in $v_2$ is 2, and the nonnegative difference between the numbers of $0$s in $v_3$ and in $v_4$ is 2. Also, either $|v_1|=|v_2|$ or $|v_3|=|v_4|$. Since all these words occur in $V$, we have at least one pair of words of the same length whose numbers of $1$s differ by $2$, contradicting the assumption that $V$ is balanced.

The case where $11$ is a subword of $V$ is similar.
\end{proof}

\begin{lem}\label{lem:VW2} If $V\in\{0,1\}^\mathbb{N}$ is a recurrent, balanced, non-eventually periodic word and $00$ is a subword of $V$, then there exist a recurrent, balanced, non-eventually periodic word $W\in\{0,1\}^\mathbb{N}$ such that $00$ is a subword of $W$, $n_0,n_1\in\mathbb{N}$, and finite words $v_0,v_1$ that end with $0$, such that $$V=v_{W{(0)}}1^{n_{W{(0)}}}v_{W{(1)}}1^{n_{W{(1)}}}v_{W{(2)}}1^{n_{W{(2)}}}\cdots.$$ Moreover, if $1$ is not a subword of $v_0$ or $v_1$, then $n_0\ne0$ or $n_1\ne0$; when $V$ begins with $0$, $v_0,v_1$ begin and end with $0$.
\end{lem}

\begin{proof} We consider several cases.

Case (1): $V$ begins with 1. Let $W$ and $a_0,a_1\in\mathbb{N}^+$ be obtained by Lemma~\ref{lem:VW}. Thus
$$ V=10^{a_{W{(0)}}}10^{a_{W{(1)}}}10^{a_{W{(2)}}}1\cdots.$$
 Let $v_0=10^{a_0}$ and $v_1=10^{a_1}$. Let $n_0=n_1=0$. Then we have $$V=v_{W{(0)}}1^{n_{W{(0)}}}v_{W{(1)}}1^{n_{W{(1)}}}v_{W{(2)}}1^{n_{W{(2)}}}\cdots $$
 as desired.

Case (2): $V$ begins with 0. Let $W$, $a_0,a_1\in\mathbb{N}^+$ and $b\in \mathbb{N}$ be obtained by Lemma~\ref{lem:VW}. Suppose $\{a_0,a_1\}=\{n,n+1\}$. Then $0<b\leq n+1$.

Subcase (2.1): $b<n$. In this subcase let $v_0=0^b10^{a_0-b}$ and $v_1=0^b10^{a_1-b}$. Let $n_0=n_1=0$. Then
$$\begin{array}{rcl} V&=&0^b10^{a_{W{(0)}}}10^{a_{W{(1)}}}10^{a_{W{(2)}}}1\cdots \\
&=& 0^b10^{a_{W{(0)}}-b}0^b10^{a_{W{(1)}}-b}0^b10^{a_{W{(2)}}-b}0^b1\cdots \\
&=&v_{W{(0)}}1^{n_{W{(0)}}}v_{W{(1)}}1^{n_{W{(1)}}}v_{W{(2)}}1^{n_{W{(2)}}}\cdots.
\end{array} $$

Subcase (2.2): $b=n$ and $a_0=n$. In this subcase let $v_0=0^n$ and $v_1=0^n10$. Let  $n_0=1$ and $n_1=0$. Then
$$\begin{array}{rcl} V&=&0^n10^{a_{W{(0)}}}10^{a_{W{(1)}}}10^{a_{W{(2)}}}1\cdots \\
&=& (0^n10)0^{a_{W{(0)}}-1}10^{a_{W{(1)}}}10^{a_{W{(2)}}}1\cdots \\
&=&v_{W{(0)}}1^{n_{W{(0)}}}v_{W{(1)}}1^{n_{W{(1)}}}v_{W{(2)}}1^{n_{W{(2)}}}\cdots.
\end{array} $$

Subcase (2.3): $b=n$ and $a_1=n$. In this subcase let $v_0=0^n10$ and $v_1=0^n$. Let $n_0=0$ and $n_1=1$. Then
$$\begin{array}{rcl} V&=&0^n10^{a_{W{(0)}}}10^{a_{W{(1)}}}10^{a_{W{(2)}}}1\cdots \\
&=& (0^n10)0^{a_{W{(0)}}-1}10^{a_{W{(1)}}}10^{a_{W{(2)}}}1\cdots \\
&=&v_{W{(0)}}1^{n_{W{(0)}}}v_{W{(1)}}1^{n_{W{(1)}}}v_{W{(2)}}1^{n_{W{(2)}}}\cdots.
\end{array} $$

Subcase (2.4): $b=n+1$. This subcase requires a different treatment. Consider the word $V'=1V$. It is easy to check that $V'$ is a recurrent, balanced, non-eventually periodic word such that $00$ is a subword of $V'$. Let $W, a_0, a_1$ be obtained by Lemma 2 for $V'$. Let $v_0=0^{a_0}$ and $v_1=0^{a_1}$. Let $n_0=n_1=1$. Then

$$\begin{array}{rcl} V'=1V&=&10^{a_{W{(0)}}}10^{a_{W{(1)}}}10^{a_{W{(2)}}}1\cdots \\
&=&1v_{W{(0)}}1^{n_{W{(0)}}}v_{W{(1)}}1^{n_{W{(1)}}}v_{W{(2)}}1^{n_{W{(2)}}}\cdots.
\end{array} $$

\end{proof}

\begin{thm} \label{thm:Stu}
If $V\in\{0,1\}^\mathbb{N}$ is a recurrent, balanced, non-eventually periodic word beginning with $0$, then $V$ has a proper spacer rank-two construction.
\end{thm}

\begin{proof} Assume first $00$ is a subword of $V$. We inductively define a sequence of infinite words $(V_i)_{i\in\mathbb{N}}$, sequences of finite words $(u_i)_{i\in\mathbb{N}}$ and $(v_i)_{i\in\mathbb{N}}$, and $\{0,1\}$-sequences $(m_i)_{i\in\mathbb{N}}$ and $(n_i)_{i\in\mathbb{N}}$ as follows. Let $V_0=V$. For $i\ge 0$, if we have defined $V_i$, let $V_{i+1}$ denote the word $W$ obtained from $V_i$ by Lemma~\ref{lem:VW2}. For each $i\geq 0$, let $u_i$, $v_i$, $m_i$, $n_i$ be, respectively, the finite words $v_0$, $v_1$ and the $\{0,1\}$-bits $n_0$, $n_1$ obtained from $V_i$ by Lemma 3. In addition, we denote $s_i=|u_i|$ and $t_i=|v_i|$.

We inductively define finite words $w_{i,1}, w_{i,2}$ for $i\geq 0$ and $\{0,1\}$-bits $p_{i,1}$, $p_{i,2}$ for $i\geq 1$ in the following. The words $w_{i,1}$, $w_{i,2}$ will give a proper spacer rank-two construction for $V$.

Define
$$\begin{array}{ll}
w_{0,1}=w_{0,2}=0, & \\
w_{1,1}=u_0,\ w_{1,2}=v_0, & p_{1,1}=m_0,\ p_{1,2}=n_0,
\end{array}
$$
and for $i\ge1$, $$\begin{array}{l}
w_{i+1,1}=w_{i,{u_i}_{(0)}+1}1^{p_{i,{u_i}_{(0)}+1}}w_{i,{u_i}_{(1)}+1}1^{p_{i,{u_i}_{(1)}+1}}\cdots w_{i,{u_i}_{(s_i-1)}+1}(1^{p_{i,1}}w_{i,2})^{m_i}, \\ w_{i+1,2}=w_{i,{v_i}_{(0)}+1}1^{p_{i,{v_i}_{(0)}+1}}w_{i,{v_i}_{(1)}+1}1^{p_{i,{v_i}_{(1)}+1}}\cdots w_{i,{v_i}_{(t_i-1)}+1}(1^{p_{i,1}}w_{i,2})^{n_i},\\
p_{i+1,1}=p_{i,m_i+1},\ p_{i+1,2}=p_{i,n_i+1}.
\end{array}$$

We claim that for any $i\geq 1$, if ${V_{i}}{(0)}=0$ then $w_{i,1}1^{p_{i,1}}$ is an initial segment of $V$ and if ${V_{i}}{(0)}=1$ then $w_{i,2}1^{p_{i,2}}$ is an initial segment of $V$. We prove this claim by induction on $i\geq 1$. For $i=1$ this follows from Lemma~\ref{lem:VW2}. In fact, in the notation of Lemma~\ref{lem:VW2} we know that $V=V_0$ has $v_{{V_1}{(0)}}$ as its initial segment, which is $u_0=w_{1,1}$ if ${V_1}{(0)}=0$ and is $v_0=w_{1,2}$ if ${V_1}{(0)}=1$. In general, suppose the claim is true $i\geq 1$. We prove the claim for $i+1$. Suppose first ${V_{i+1}}{(0)}=0$. Then in the notation of Lemma~\ref{lem:VW2}, $v_{{V_{i+1}}{(0)}}=u_i$, and therefore $u_i1^{m_i}$ is an initial segment of $V_i$. From the inductive hypothesis we know that $w_{i,{u_i}{(0)}+1}1^{p_{i,{u_{i}}{(0)}+1}}$ is an initial segment of $V$.
Now let $U$ be the infinite word such that 
$$ V=w_{i,{u_i}{(0)}+1}1^{p_{i,{u_{i}}{(0)}+1}}U. $$
Note that $V_i={u_i}{(0)}U_i$, and thus we may apply the inductive hypothesis to ${U_i}(0)$ to conclude that $w_{i,{u_i}{(1)}+1}1^{p_{i,{u_{i}}{(1)}+1}}$ is an initial segment of $U$. Repeating this for every digit of $u_i1^{m_i}$, we conclude that $w_{i+1,1}$ is an initial segment of $V$. It follows easily that $w_{i+1,1}1^{p_{i+1,1}}$ is an initial segment of $V$. The argument for the case ${V_{i+1}}{(0)}=1$ is similar. The claim is thus proved. 

From the claim it follows that for all $i\geq 1$, either $w_{i, 1}$ or $w_{i,2}$ is an initial segment of $V$. Renaming $w_{i,1}$ and $w_{i,2}$ so that for all $i\geq 1$ we always have $w_{i,1}$ is an initial segment of $V$. The resulting sequence $\{w_{i,j}\}_{i\in\mathbb{N},1\le j\le2}$ is a proper spacer rank-two construction of $V$.

Next assume $11$ is a subword of $V$. By Lemma~\ref{lem:VW}, there exist a recurrent, balanced, non-eventually periodic word $W\in\{0,1\}^\mathbb{N}$ such that 00 is a subword of $W$, and $a_0,a_1\in\mathbb{N}^+$ such that $|a_0-a_1|=1$, such that $$V=01^{a_{W(0)}}01^{a_{W(1)}}01^{a_{W(2)}}0\cdots.$$ 
Similar to the first part of this proof, we define inductively $(V_i)_{i\in\mathbb{N}}$, $(u_i)_{i\in\mathbb{N}}$, $(v_i)_{i\in\mathbb{N}}$, $(m_i)_{i\in\mathbb{N}}$ and $(n_i)_{i\in\mathbb{N}}$ as follows. Let $V_0=W$. For $i\ge 1$, let $V_{i+1}$ denote the word $W$ obtained from $V_i$ by Lemma~\ref{lem:VW2}. For each $i\geq 0$, let $u_i$, $v_i$, $m_i$, $n_i$ be, respectively, the finite words $v_0$, $v_1$ and the $\{0,1\}$-bits $n_0$, $n_1$ obtained from $V_i$ by Lemma 3. In addition, we denote $s_i=|u_i|$ and $t_i=|v_i|$.

Define 
$$ w_{0,1}=w_{0,2}=0,\ \ \  p_{0,1}=a_0,\ p_{0,2}=a_1.$$
For $i\ge0$, define $$\begin{array}{l}
w_{i+1,1}=w_{i,{u_i}{(0)}+1}1^{p_{i,{u_i}{(0)}+1}}w_{i,{u_i}{(1)}+1}1^{p_{i,{u_i}{(1)}+1}}\cdots w_{i,{u_i}{(s_i-1)}+1}(1^{p_{i,1}}w_{i,2})^{m_i}, \\ w_{i+1,2}=w_{i,{v_i}{(0)}+1}1^{p_{i,{v_i}{(0)}+1}}w_{i,{v_i}{(1)}+1}1^{p_{i,{v_i}{(1)}+1}}\cdots w_{i,{v_i}{(t_i-1)}+1}(1^{p_{i,1}}w_{i,2})^{n_i},\\
p_{i+1,1}=p_{i,m_i+1},\ p_{i+1,2}=p_{i,n_i+1}.
\end{array}$$ 
Then by a similar argument as above, for each $i\geq 0$, either $w_{i,1}$ or $w_{i,2}$ is an initial segment of $V$. After renaming $w_{i,1}$ and $w_{i,2}$ so that $w_{i,1}$ is always an initial segment of $V$, we get that $\{w_{i,j}\}_{i\in\mathbb{N},1\le j\le2}$ is a proper spacer rank-two construction of $V$.
\end{proof}

We now show that the system determined by a Sturmian sequence cannot be (spacer) rank-one, which together with Theorem~\ref{thm:Stu} will show that all Sturmian systems are spacer rank-two. We note that it was known that a measurable dynamical system for some Sturmian sequences is not rank one (\cite[Proposition 5]{Ch00}, \cite[Theorem 3]{FH}, \cite[Chapter 6]{Fo02}).

\begin{prop}
Let $(X, \sigma)$ denote the shift system for a  Sturmian sequence. Then $(X, \sigma)$ is not rank one.
\end{prop}
\begin{proof}
Let $W$ denote the Sturmian sequence associated with $X$ and first assume it is type $0$. Since all Sturmian sequences are balanced, between any two $1$'s there can be either $n$ or $n + 1$ $0$'s. Let $V$ be a rank-one word such that $X_V = X$. We know both $V$ and $W$ are recurrent. Since any subword of $V$ is a subword of $W$, $V$ must satisfy the same properties that the sequence is balanced. Let $V$ start with $0^k 1$ with $0 < k \le n + 1$. Suppose $V$ is a rank-one word built by some word $p$ which contains $0^k 1$ and ends with a $0$. If $k = n + 1$, then $V$ cannot contain two $p$'s in a row since otherwise, there would be $n + 2$ zeroes between any two ones, contradicting the fact that $V$ is balanced. Thus, $V$ must be of the form $p1p1p1p\dots$, contradicting nonperiodicity of $W$. If $p$ ends with $10^{n + 1}$, then the next symbol in $W$ must be a $1$ then followed by a $0$. Thus, $V$ is of the form $p1p1p1\dots$, which obviously cannot be the case since $W$ is not periodic. If $p$ ends with $10^n$ and the next digit is a $1$, then $V$ contains the sequence $10^n10^k1$ which cannot happen if $k < n$, but if $k = n$ only possibly occurs if $10^{n + 1}10^{n + 1}1$ does not occur. If $n = 1$, then $10^n10^n1$ can't happen by the hypothesis of the theorem. If $n > 1$, then $pp$ can't occur in $V$ since otherwise, there would be $2n > n + 1$ zeroes in a row. Finally, if $p$ ends with less than $n$ zeroes then $V$ must be of the form $ppppp\dots$, contradicting nonperiodicity of Sturmian sequences. 

In the case of $n = 1$, we show via induction that such a $p$ must be of the form $01010101\cdots10$. Note that both $pp$ and $p1p$ must appear in $V$ since otherwise there would be periodicity. In addition, by the argument above, $p$ must begin with $01$ and since $pp$ appears, it must end with $10$. Since $p1p$ appears, $10101$ appears and therefore since $V$ is a Sturmian sequence, $1001001$ cannot appear. Thus, $p$ must end with $01010$ so $1010101$ appears in the Sturmian sequence. Now suppose $10^{n}$ appears in the Sturmian sequence where the length of $(10)^{n}$ is less than that of $p$. Then $100(10)^{n - 2}1001$ cannot appear in the sequence and since $pp$ appears, $p$ must end with $(10)^n$ and since $p1p$ appears, the Sturmian sequence contains $(10)^{n + 1}$. By induction, $p$ must be of the form $010101\cdots10$. Since $V$ is rank-one, this would mean that the first $|p|$ letters of $V$ are $p$ for infinitely many $p$ of the form $01010\dots10$ so $V$ would be periodic. This is a contradiction.

We next consider the case when $W$ is type 1. Suppose there exists a rank-one $V$ such that $X_V = X$. As in the previous case, we argue that between two zeroes in $V$ and $W$, only $1^n$ and $1^{n + 1}$ appear for some $n \ge 1$. In addition, as before, both $V$ and $W$ must be balanced sequences that cannot be eventually periodic (making $v$ a Sturmian sequence as well). Suppose $p$ builds $V$. Note that $p$ must begin and end in $0$. Both $p1^np$ and $p1^{n + 1}p$ must appear in $v$ because otherwise, $v$ would be periodic. 

If $p$ begins with $01^n$ and ends with $1^{n + 1}0$, then since $p1^np$ and $p1^{n + 1}p$ appears in $V$, $1^{n + 1}01^{n + 1}$ and $01^n01^n0$ appears in $V$, contradicting the fact that $V$ is balanced. By a similar argument, if $p$ begins with $01^{n + 1}$ and ends with $1^n0$, the word $01^n01^n0$ and $1^{n + 1}01^{n + 1}$ must appear in $v$, contradicting the fact that $v$ is balanced. \\\\
If $p$ starts with $01^n$ and ends with $1^n0$, since $p1^np$ and $p1^{n + 1}p$ both appear in $v$, $01^n01^{n + 1}01^n0$ and $01^n01^n1^n0$ both appear in $v$. We show via induction that $p$ must be of the form $01^n01^n \dots 1^n0$. If $p$ starts with $01^n01^{n + 1}$, then $1^{n + 1}01^n01^{n + 1}$ and $01^n01^n01^n0$ appear in $v$, contradicting balancedness of $v$. Thus, $p$ must begin with $01^n01^n0$. Now suppose $p$ begins with $(01^n)^k0$. If $p$ begins with $(01^n)^k01^{n + 1}0$, then since $p1^np$ and $p1^{n + 1}p$ both appear in $v$, the string $1^{n + 1}(01^n)^k01^{n + 1}$ and $(01^n)^{k + 2}0$ both appear in $v$. They are both strings of length $(n + 1)(k + 2) + 1$ but the latter has $n(k + 2)$ $1$'s while the former has $n(k + 2) + 2$ $1$'s, a contradiction to $v$ being balanced. Since $v$ is rank-one, it must be built with arbitrarily many $p$, so it must begin with $(01^n)^k$ for $k$ arbitrarily large: i.e. it must be periodic. This is a contradiction. 

If $p$ starts with $01^{n + 1}$ and ends with $1^{n + 1}0$, we argue similarly that $p$ must be of the form 
$$01^{n + 1}01^{n + 1}\dots 1^{n + 1}0.$$ 
Suppose $p$ begins with $(01^{n + 1})^k01^n0$. Then since $V$ contains $p1^np$ and $p1^{n + 1}p$, $V$ must contain the string $1^{n + 1}01^{n + 1}(01^{n + 1})^k$ and must contain $01^n(01^{n + 1})^k01^n0$, both of which are length $(n + 2)(k + 1) + n + 1$ but the former has $(n + 1)(k + 3)$ $1$'s while the latter has $(n + 1)(k + 3) - 2$ $1$'s, a contradiction to $V$ being balanced. Hence, $V$ begins with $(01^{n + 1})^k$ for $k$ arbitrarily large, so it is periodic. This is a contradiction. 
\end{proof}

\begin{cor} \label{C:sturmr2} If $V\in\{0,1\}^\mathbb{N}$ is a Sturmian sequence, then $(X_V,\sigma)$ is a spacer rank-two subshift.
\end{cor}

\begin{proof} If $v$ is an initial segment of $V$ and $W$ is such that $V=vW$, then $X_V=X_W$ since $V$ is recurrent, and thus $W$ is also Sturmian. Thus without loss of generality we may assume that $V$ starts with $0$. By Theorem~\ref{thm:Stu}, $X_V$ is a spacer rank-two subshift.
\end{proof}

\begin{example} \rm{
The Fibonacci substitution $0 \mapsto 01$ and $1 \mapsto 0$ does not contain $10101$ and is a type $0$ Sturmian sequence. Hence, it is not rank one, and since it is a substitution sequence, it must be spacer rank-two. We know it defines a spacer rank-two system.
}\end{example}

\section{Characterizations and Additional Examples of Spacer Rank-Two Systems}\label{S:charrank2}

In this section we give conditions for a symbolic subshift to be spacer rank-two, a classification for spacer rank-two systems (Theorem \ref{thm:rank two builds rank one}), and  some examples, including a class that generalizes the Morse sequence and spacer rank-two Sturmian sequences. In this context we note that del Junco showed in \cite{Ju77} that the finite measure-preserving Morse transformation is rank-two (as a measure-preserving system), which implies that the Morse system is rank-two, but our proof is independent of del Junco's and applies to a larger class.


We note that when dealing with an infinite or bi-infinite word that is built from a finite word $v$, an occurrence of $v$ in the word at position $i$ does not necessarily imply that there is an occurrence of $v1^a$ for some $a \geq 0$ ending at position $i-1$. For example, consider a rank-one word $V$ defined as
\begin{align*}
v_0 &= 0,\\
v_1 &= 00100,\\
  v_{n+1} &= v_n 1 v_n 11 v_n \ \text{ for } n > 0,
\end{align*}
and the system $(X_V, \sigma)$ associated to $V$. Suppose we see an occurrence of $v_1$ in a bi-infinite word $x = \cdots \underline{0}0100 \cdots$ at the $0^\text{th}$ position for some $x \in X_V$. Knowing that $v_1$ builds $x$, one might think that $v_1$ at the $0^\text{th}$ position must be preceded by the word $v_1 1^a$, for some $a \geq 0$, but it is possible that the occurrence of $v_1$ at the $0^\text{th}$ position is in fact part of an occurrence of $v_2$:
$$\cdots 00100 \ 1 \ 00100 \ 11 \ 001\underline{0}0 \ 1 \ 00100 \ 1  \ 00100 \ 11 \ 00100 \cdots$$


As in \cite{GaHi16AT,GaZi18,Ka84}, given an infinite word $V$, if there exists a unique decomposition of $V$ of the form
$$V = v 1^{a_1} v 1^{a_2} v 1^{a_3} v 1^{a_4} v \cdots$$
such that $a_i \geq 0$ for all $i \in \N^+$, we say each occurrence of $v$ shown above is an \textbf{expected occurrence}. Similarly, for any bi-infinite word $x$, if there exists a unique decomposition of $x$ to the form
$$x = \cdots v 1^{a_{-2}} v 1^{a_{-1}} v 1^{a_0} v 1^{a_1} v 1^{a_2} v \cdots$$
such that $a_i \geq 0$ for all $i \in \mathbb{Z}$, then each occurrence of $v$ shown above is an \textbf{expected occurrence}.

Kalikow showed in \cite{Ka84} that whether an occurrence of $v$ is expected in a bi-infinite word can be resolved uniquely for  aperiodic bi-infinite words.

\begin{lem}[Kalikow \cite{Ka84}] \label{kalikow}
Given a bi-infinite word that is built from a finite word $v$, if the entire word is aperiodic, then there is a unique way to decompose the word into expected occurrences of $v$'s separated by $1$'s.
\end{lem}
Note that for any infinite word that is built from a finite word $v$, the decomposition of the word into expected occurrences is unique because the first occurrence of $v$ has to be an expected occurrence. We also point out that the case which the word is periodic and thus cannot be decomposed uniquely into expected occurrences of $v$ is trivial, since the word has to be at most (spacer) rank-one.

\begin{lem}
Let $V$ be an infinite word. If $V$ starts with $0$ and is periodic, then $V$ is rank-one.
\end{lem}

\begin{proof}
Let $k$ be the period of $V$. Since $V$ starts with $0$, the first digit of $V$, $V_{(1)} = 0$. Define
\begin{center}
\begin{itemize}
    \item[] $v_{0,1} = 0$
    \item[] $v_{n,1} = V{(1)} V{(2)} \cdots V{(a_n)} \text{ for } n > 0$
\end{itemize}
\end{center}
where $a_n$ is the position of the last $0$ within the first $nk$ digits of $V$. Clearly, each $v_{i,1}$ builds $v_{i+1,1}$, and 
$V \upharpoonright |v_{i,1}| = v_{i,1}$. Then $V$ is rank-one.
\end{proof}

\begin{cor}
Let $V$ be an infinite word. If $V$ is periodic and $V \neq 111\cdots{}$, then $X_V$ is rank-one.
\end{cor}

\begin{proof}
The corollary follows from {Corollary \ref{shift infinite word}} which shows that infinite words up to shifts have the same rank-one system. Given $V$ periodic and $V \neq 111\cdots{}$, we can remove the leading digits of $V$ up to the first $0$ to get an infinite periodic word that starts with $0$.
\end{proof}
We will need  the following lemma.
 
\begin{lem}[Gao--Hill \cite{GaHi16AT}]
Suppose $V$ is a rank-one word and $(X, \sigma)$ is the rank-one system associated to $V$. If $x \in X$ contains an occurrence of $0$, then $x$ contains an occurrence of every finite subword of $V$.
\end{lem}
We extend this lemma to systems of spacer rank $n$ for any $n$.
\begin{lem} \label{every subword of W is a subword of x}
Let $W$ be an infinite word that is spacer rank-$n$ and let $(X_W, \sigma)$ be the system associated to $W$. For any $x \in X_W$, if $x \neq \cdots 11\underline{1}11 \cdots$, then every finite subword of $W$ is a subword of $x$.
\end{lem}

\begin{proof}
Since $W$ is at most spacer rank-$n$ there exists an infinite sequence
$$\{w_{i, 1}, w_{i, 2}, \cdots, w_{i, n} \}_{i \in \mathbb{N}}$$
of sets of finite words such that $W$ is built from $\{w_{i, 1}, w_{i, 2}, \cdots, w_{i, n} \}$ for any $i$. For any finite subword $u$ of $W$, $u$ is a subword of $w_{k,1}$ for some $k \geq 0$. If $x \neq \cdots 11\underline{1}11 \cdots$, then $x$ contains an occurrence of $0$. Every occurrence of $0$ in $x$ is a part of an occurrence of $w_{l,1}$ for some $l \geq 0$. Since both $w_{l,1}$ and $w_{k,1}$ has to be an occurrence of $w_{\max(l, k),1}$, if $x$ contains an occurrence of $0$, then $x$ contains an occurrence of $w_{\max(m+1, k),1}$. Therefore, $x$ contains an occurrence of $u$.
\end{proof}


The following lemma will be used to prove that certain systems are spacer rank-two.


\begin{lem}\label{lem:unique-decomposition-rank-2}
Suppose $V$ has a spacer rank-$n$ construction built by $V_{k,i}$ for all $k$ and $1 \leq i \leq n$ and there exists some finite word $v$ that builds $x \in X_V$. 
Suppose there exist two distinct ways to decompose $V_{k,1}$ for all $k > N$ for some $N$ where the decomposition is into three words $a, b_k, c_k$ or $a' ,b'_k, c'_k$ with the following properies

\begin{enumerate}
\item  $V_{k,1} = a b_k c_k = a' b'_k c'_k$, 
\item $|a| < v$, $|a'| < v$, $|c_k| < v$, $|c'_k|< v$, 
\item $b_k$ and $b_k'$ are built from $v$,
\item there exits finite words $d_k, d'_k, e_k, e'_k$ such that $d_k a = v$, $d'_ka' = v$, $c_ke_k = v$ and $c'_k e'_k = v$.
\end{enumerate}
Then there exists $x \in X_V$ that is periodic.
\end{lem}

\begin{proof}
Suppose there exist two distinct ways to decompose $V_{k, 1}$ for every $k > N$ for some integer $N$ satisfying the conditions stated in the lemma. Consider the following two sequences of bi-infinite words $\{y_i\}_{i \in \mathbb{N}}$ and $\{z_i\}_{i \in \mathbb{N}}$:
\begin{itemize}
    \item[] $y_i = \cdots vvv \ d_k V_{k, 1} e_k \ vvv \cdots$ such that $V_{k, 1}$ starts at position $-k$ for $k = i + N$
    \item[] $z_i = \cdots vvv \ d'_k V_{k, 1} e'_k \ vvv \cdots$ such that $V_{k, 1}$ starts at position $-k$ for $k = i + N$
\end{itemize}
From the compactness of $\{0, 1\}^\mathbb{Z}$, we know that there exists $S \subset \mathbb{N}$ such that $\{y_i\}_{i \in S}$ is a convergent sub-sequence of $\{y_i\}_{i \in \mathbb{N}}$. Let $y$ be the limit of $\{y_i\}_{i \in S}$. Note that $y \in X_V$ since every finite subword of $y$ is a subword of $M_{k, 1}$ for any $k$ and thus is a subword of $V$. For any $i$, both $z_i$ and $y_i$ have $V_{k, 1}$ at the same position around $0$, with $k$ increasing as $i$ increases. Thus,
$$\lim_{i\rightarrow\infty} d(z_i, y_i) = 0$$
so $\{z_i\}_{i \in S}$ also converges to the same limit $y$. However, the positions of expected occurrences of $v$ in $z_i$ and $y_i$ for every $i$ are different, so there exist two distinct ways of decomposing $y$ into expected occurrences of $v$. Thus, by {Lemma \ref{kalikow}}, $y$ must be periodic.
\end{proof}
For clarity in proving that $X_V$ is not rank one and to avoid excessive notations, for any finite words $a$ and $b$, we denote an occurrence of $a$ separated by some number (which could be zero) of $1$'s as $a1^*b$.

Note that Theorem~\ref{T:rank2sub} only shows that the words that are fixed points of proper constant-length substitutions are spacer rank-two, and that it does not show that the systems generated by those words are spacer rank-two. 
Along those lines, we have the following theorem which does indeed prove that some systems are spacer 
rank-two as long as they are generated by words that satisfy some generalization of the properties satisfied by the Morse sequence. 

\begin{thm}\label{thm:TM-rank-2-extension}
Suppose that $W$ has a spacer rank-two construction such that $\displaystyle W=\lim_{k \to \infty} w_{k,1} $, and $W$ is built from $\{w_{k,1},w_{k,2}\}$ for all $k$. If for all $k$, the $w_{k,i}$ satisfy the following properties:

\begin{enumerate}
\item There exists $M \in \N$ such that the length of the longest prefix or suffix shared by $w_{k,1}$ and $w_{k,2}$ is less than $M$ for all $k$. 
\item There exists $p$ such that $W$ is free of all $p$-powers
\end{enumerate}
then $(X_W,\sigma)$ is not a (spacer) 
rank-one system, and so $(X_W,\sigma)$ is a spacer 
rank-two system.
\end{thm}

\begin{proof}
Suppose that $(X_W,\sigma)$ is a (spacer) 
rank-one system. 
Then there exists some rank-one word $V$ such that $X_V = X_W$. 
Let $v$ be any finite subword of $V$ that builds $V$. 
We can ensure that $|v|$ is sufficiently high such that it does not build $W$ because there are only a finite number of words that build $W$ by themselves since $W$ is spacer rank-two. 
For any $x \in X_W$, we will have that $x$ is not periodic because we stipulated that $W$ is free of all $p^\text{th}$ powers, so there is some maximum number of times a subword can be repeated. 
By Lemma \ref{every subword of W is a subword of x}, we have that every finite subword of $V$ is a finite subword of $x$ and every finite subword of $W$ is a finite subword of $x$. 
So we can break $x$ up into expected occurrences of $v$ separated by $1$s or for each $k$, into expected occurrences of $W_{k,1}$ and $W_{k,2}$ with spacers. 
Since every subword of $W$ is a subword of $x$, there are infinitely many occurrences of subwords of $x$ of the form $W_{k,1}1^* W_{k,1}$ and $W_{k,2}1^*W_{k,1}$ for each $k$. 
Because $W$ is of spacer rank two and $|v|$ can be made large enough that it does not fit in $W_{k,1}$, we have that no expected occurrence of $W_{k,1}$ in $x$ begins with an expected occurrence of $v$.

By Lemma \ref{lem:unique-decomposition-rank-2}, we have that since $x$ is aperiodic we have a unique decomposition $W_{k,1} = abc$ where $|a|, |c| <|v|$, $b$ is built from $v$ and there exists $d, e$ such thaat $d a = v$ and $c e = v$. We will let the end of $W_{k,2}$ be $f$. 
Then we have that $v = c1^*a = f1^*a$ because of the placement of $v$ in $W_{k,1}1^* W_{k,1}$ and $W_{k,2}1^*W_{k,1}$. 
Since $W_{k,2}$ ends with $0$, we have that $f$ ends with $0$, so $f = c$ which has some length bounded by $M$ by the assumptions we have about shared prefixes and suffixes of $W_{k,1}$ and $W_{k,2}$.

Now, we have either that every expected occurrence of $w_{k,2}1^*w_{k,2}$ starts and ends with $v$ or there is a $k$ such that some expected occurrence of some $w_{k,2}1^*w_{k,2}$ does not start and end with $v$.
In the first case, consider any expected occurrence of \[w_{k,1}1^*w_{k,2}1^*w_{k,2}1^*w_{k,2}1^* \cdots 1^*w_{k,1}\] with $p-1$ copies of $w_{k,2}$ where $p$ is the lowest bound on $p^\text{th}$ powers in $W$. 
Then the first $w_{k,2}$ starts with $v$ and the second ends with $v$. 
Since $W$ has no $p^\text{th}$ powers for some $p$, there do not exist $p$ expected  occurrences of $W_{k,2}$ in a row, so we must have that there is a copy of $w_{k,2}1^*w_{k,2}$ folllowed by $w_{k,1}$. 
But this means that that copy of $w_{k,1}$ starts with an expected occurrence of $v$, so we have that $c1^*d = v$ which implies that $a = d$ which bounds the length of $a$ and $d$ above by $M$.

Similarly, if there exists and expected occurrence of $w_{k,2}1^*w_{k,2}$ that does not start and end with $v$, then there is an expected occurrence of $w_{k,2}1^*w_{k,2}$ such that $f1^*d = v$ so $a = d$ which also bounds above their length by $M$. 
So $v = a1^*c$, which means we can bound above the length of $v$ by $2M + p$ which means there cannot be infinitely many words $v$ building $V$ which contradicts that $V$ is rank-one. 
So we cannot have that $(X_V,\sigma)$ is a rank-one system which means it is a spacer rank-two system.
\end{proof}

\begin{cor}
The shift system defined by the Morse sequence is spacer rank-two.
\end{cor}
\begin{proof} One can verify by induction that in the definition of the Morse sequence, for any $i \in \mathbb{N}$, the longest subword that both $M_{i,1}$ and $M_{i,2}$ start with is $0$, and the longest subword that both $M_{i,1}$ and $M_{i,2}$ end with is also $0$, and it is well-known that the Morse sequence has no cubes. 
\end{proof}

In Theorem \ref{thm:rank two builds rank one} and Theorem \ref{thm:class-rank-two} we provide a classification of which words with a spacer rank-two construction generate rank-one systems, and which spacer rank-two words generate spacer rank-two systems. 
 

We use the following lemma from \cite{GaHi16AT} . 

\begin{lem}\label{lem:two-sided-words-built-by-v}
If $V$ is built by $v$, then each $x \in X_V$ is built by $v$.
\end{lem}

The proof relies on the concept of expected occurrences of $v$ in $V$ and in $x$. 
Using Lemma \ref{lem:two-sided-words-built-by-v}, we can begin our classification of which spacer rank-two words build rank-one systems.

\begin{thm}\label{thm:rank two builds rank one}
If $V$ be a word with a spacer rank-two construction with levels $\{v_{n,1}, v_{n,2}\}$,  such that  the first word in the construction $\{v_{n,1}, $  appears only once in $V$, for some $n>1$,  
then $(X_V, \sigma)$ is a rank-one system.
\end{thm}

\begin{proof}
Suppose first that $(X_V, \sigma)$ is a spacer rank-two system and that $V$ is built by $\{P_n, w_n\}$ with $\displaystyle \lim_{n \to \infty} P_n = V$ and $P_n$ only having one expected appearance (or finitely many times, it is equivalent) in $V$.  
By our definition of spacer rank-two words, we have that $P_{n+1}$ is built by $P_n$ and $w_n$, and $w_{n+1}$ is built by $P_n$ and $w_n$. 
Note that $P_{n+1}$ cannot be built by only $P_n$ since $V$ is not a rank-one word, so we in fact have that $w_{n+1}$ is built only by $w_n$, as otherwise $P_n$ would appear infinitely often . 
The sequence $(w_n)_{n \geq 1}$ is a generating sequence for some rank-one word we will call $W = \displaystyle \lim_{n \to \infty} w_n$. 
We claim that $X_V = X_W$, so $X_V$ would be a rank-one system.

By Corollary \ref{cor:equivalent-systems} we know that $X_V = X_W$ if and only if $V$ and $W$ have the same subwords that appear infinitely often. 
Suppose $u$ appears infinitely often in $W$. 
Then there exists some $n$ such the $u$ is a subword of $w_n$, so $u$ appears infinitely often in $V$ because $w_n$ does.
Conversely, suppose $u$ appears infinitely often in $V$. 
Then $u$ must be a subword of $w_n$ for some $n$ because for any $n$, there will be copies of $u$ outside of $P_n$, and we have that $|w_n| \to \infty$, so for a large enough $n$, $u$ must be a subword of $w_n$. 
Then we have that $u$ appears infinitely often in $W$. 
Thus $X_V = X_W$, which means that $X_V$ is a rank-one system, a contradiction. 

\end{proof}

\begin{thm}\label{thm:class-rank-two}
Let $V$ be a spacer rank-two word such that for all spacer rank-two constructions of $V$, the first word in the construction appears infinitely many times. 
Then $(X_V, \sigma)$ is a spacer rank-two system.
\end{thm}

\begin{proof}
Suppose now that $V$ is a spacer rank-two word such that every spacer rank-two construction of $V$ by $\{v_{n,1}, v_{n,2}\}$, $\{v_{n,1}$ appears infinitely many times. For the sake of contradiction suppose that the system $X_V$ is rank-one. 
Then there exists some rank-one word $W$, such that $X_W = X_V$. 

Let us define the word $\overline{V}$ by $\overline{V}{(i)} = \begin{cases} V{(i)} & i \geq 0 \\ 1 & i < 0 \end{cases}$, 
and consider the sets \[S_n = \{ \sigma^k(\overline{V}) : \sigma^k(\overline{V}) \upharpoonright |v_{1,n} |= v_{1,n} \}.\] 

Since $\{0,1\}^\Z$ is compact, $S_n$ must have a convergent subsequence for each $n$. 
Pick a convergent subsequence of $S_n$ and say that it converges to $x_n$, which we one can verify is in $X_V$. 

Now, consider the sequence $x_n$ and note that each $x_n$ looks something like $\cdots .v_{1,n} \cdots$ with the $v_{1,n}$ starting at $0$. 
Since  $X_V$ is compact we must have that there is a convergent subsequence of the $x_n$ which converges to some element of $X_V$ that we call $x \in X_V$. 
The $x_n$  have $v_{1,k_n}$ as their first $|v_{1,k_n}|$ digits, so $x = \cdots .V$ with $V$ starting at $0$. 
It follows that we have a word in   $X_V$ that looks like $\cdots .V$.

By Lemma \ref{lem:two-sided-words-built-by-v} the word $x$ must be built by any $w$ building $W$ since $X_V = X_W$. 
So we have that $x = \cdots .V$ must be built by $w$ which tells us that some shift of $V$ is built by $w$. 
We do not know that the $w$'s building $x$ have to have some $w$ starting at the $0^\text{th}$ position, but we do have to have that the $w$'s continue on infinitely to the right because otherwise $V$ would have only finitely many $0$'s, which means it is not a rank-one word.

So we have that there exists some $\ell$ such that $\sigma^\ell(V)$ is built by $w$. 
This means that for any such $w$  we have that $V$ is built in a way that looks like \[V=P_w w 1^{a_1} w1^{a_2} \ldots.\]

We show that this extends to form a spacer rank-two construction that is not proper.
 We   take $w_n$ to be a generating sequence for $W$  and $P_n = P_{w_n}$. Then $w_n$ and  $P_n$  build $w_{n+1}$ and $P_{n+1}$.  We consider  three cases. 
First, all the first copies of $w_n$ start at $0$, so $P_{w_n} = \epsilon$ is empty and $V = W$. 
This is clearly a contradiction, so there must be some $w_n$ not starting at $0$. The second case is that there exists $ N$ such that all the first copies of $w_n$ start at or before $N$. 
If this is the case, then there must be some $k$ such that infinitely many of the $w_n$ start at $k$. 
We can take this subsequence of of the $w_n$ and call it $u_n$. Then we have that $P$ and $u_n$ build $V$ for some fixed $P$. 
Since we want $P \to V$ as $n \to \infty$ we instead define $P'_n = Pu_n$ which will tend to $V$ and we do indeed have $P'_{n+1} = Pu_{n+1}$ is built from $P'_n = Pu_n$ and $u_n$ because $u_{n+1}$ is built from $u_n$. 
So this extends into a complete spacer rank-two construction for $V$ where the first term in the construction appears only once. 
Finally, we have that case where there is no $N$ such that all the first copies of $w_n$ start before $N$. 
This means we can take a subsequence $u_n$ of $w_n$ such that the starting locations of the first $w_n$ in their respective construction of $V$ increase monotonically. 
So the lengths of $P$ are increasing and we do in fact have that $P \to V$ so by taking this subsequence, we have a complete spacer rank-two construction for $V$ where the first term in the construction appears only once, so not a proper construction, a contradiction. It follows that $X_V$ cannot be  a rank-one system, and since $V$ is a word with a spacer rank-two construction it means that $(X_V, \sigma)$ is a rank-two system.

\end{proof}

We have a classification of when a spacer rank-two word will generate a rank-one system and when it will generate a spacer rank-two system. 
We also have a small extension of this result that applies the results to spacer rank-$n$ words, whose proof follows immediately from the same proof as the above theorem.

\begin{prop}
Let $V$ be a spacer rank-$n$ word with $n > 1$ such that for all spacer rank-$n$ constructions of $V$, the first word in the construction appears infinitely many times. 
Then $(X_V, \sigma)$ is not a rank-one system.
\end{prop}


It is known that any rank-one system is either periodic, minimal, or has a single fixed point and is minimal upon removing that fixed point \cite[Proposition 2.4]{GaHi16AT}, so  a rank-one system has at most two orbit closures. In 
our final example we construct a system $X_V$ that has at least four orbit closures. It would be interesting to know an upper bound on the number of orbit closures depending on the spacer rank.
 
\begin{example} \rm{
We construct a proper spacer rank-two word defining a spacer rank-two system that is not minimal and has four orbit closures. 
Consider the word  $V =\displaystyle \lim_{n \to \infty} v_{n,1}$ defined by 
\begin{align*}
v_{0,1} &= 0, v_{0,2}=0,\\
 v_{n,1} &= v_{n-1,1}1^{2^n} v_{n-1,2} \\ 
 v_{n,2}&=v_{n-1,2}1v_{n-1,2}.
\end{align*} 
The word $V$ starts as follows \[V = 010110101111010101011111111010101010101010\cdots\] 
We claim  that the system $(X_V, \sigma)$ is not minimal. 
In fact, it has at least four different orbit closures. 
This also implies that $(X_V, \sigma)$ is a spacer rank-two system. 
Therefore $V$ is a spacer rank-two system as it has a word with a spacer rank-two construction that generates it. This is another method showing that the word $V$ is actually not a rank-one word.

Because $v_{n,2} = (01)^{2^n-1}0$ we have that $\ldots 010101\ldots \in X_V$ and because we have $1^{2^n}$ for spacers when building $v_n$ we have that $\ldots 1111 \ldots \in X_V$. 
The first of these two is periodic with period two and the second is a fixed point. 
Outside this we show $X_V$ has at least two orbit closures.
Consider the points $x_1 = \ldots 01010.11111\cdots$ and $x_2 = \ldots 11111.10101\ldots$. 
We claim that both of these are in $X_V$ and have different orbit closures.

First, note that $x_1 \in X_V$ because \[v_{n+1,1} = v_{n,1} 1^{2^n} v_{n,2} = v_{n-1,1} 1^{2^{n-1}} v_{n-1,2} 1^{2^n} v_{n,2}.\] 
So we have that for all $n$, $(01)^{2^{n-1}-1} 01^{2^n}$ is a subword of $V$ which means that $x_1    \in X_V,\text{ as is }\sigma(x_1) = \ldots 10101.11111\ldots \in X_V.$
Similarly, we have that $1^{2^n} v_{n,2}$ is a subword of $V$ for all $n$, so $1^{2^n}(01)^{2^n-1} 0$ is a subword of $V$ for all $n$. 
This means that $x_2  \in X_V$.

Now we must show they have different orbit closures. 
It is sufficient to show that $x_1$ is not in the orbit closure of $x_2$. 
Suppose $x_1  $ is in the orbit closure of $x_2  $. 
Then for all $\varepsilon > 0$, there exists some $z$ in the orbit of $x_2$ such that $d(x_1,z) < \varepsilon$. 
Taking $\varepsilon < 1/2^n$, this means that $x_1$ agrees with some shift of $x_2$ on the middle $-n$-{th} through $n$-{th} digits. 
However, $x_1$ has $01011$ as it's middle $5$ digits, whereas any shift of $x_2$, will either have $11111$, $11110$, $11101$, $11010$, $10101$, or $01010$ as its middle five digits. 
None of these agree with the middle $5$ digits of $x_1$, so no shift of $x_2$ can be closer than $1/2^5$ to $x_1$.
 So $x_1$ cannot be in the orbit closure of $x_2$. So $X_V$ has two different infinite orbit closures.
}\end{example}

\bibliography{bibliographyR2-V5}

\begin{thebibliography}{10}

\bibitem{AdFePe17}
Terrence Adams, S\'{e}bastien Ferenczi, and Karl Petersen.
\newblock Constructive symbolic presentations of rank one measure-preserving
  systems.
\newblock {\em Colloq. Math.}, 150(2):243--255, 2017.

\bibitem{Bo13}
J.~Bourgain.
\newblock On the correlation of the {M}oebius function with rank-one systems.
\newblock {\em J. Anal. Math.}, 120:105--130, 2013.

\bibitem{Ch67}
R.~V. Chacon.
\newblock A geometric construction of measure preserving transformations.
\newblock In {\em Proc. {F}ifth {B}erkeley {S}ympos. {M}ath. {S}tatist. and
  {P}robability ({B}erkeley, {C}alif., 1965/66), {V}ol. {II}: {C}ontributions
  to {P}robability {T}heory, {P}art 2}, pages 335--360. Univ. California Press,
  Berkeley, CA, 1967.

\bibitem{Ch69}
R.~V. Chacon.
\newblock Weakly mixing transformations which are not strongly mixing.
\newblock {\em Proc. Amer. Math. Soc.}, 22:559--562, 1969.

\bibitem{Ch00}
Nataliya Chekhova.
\newblock Covering numbers of rotations.
\newblock {\em Theoret. Comput. Sci.}, 230(1-2):97--116, 2000.

\bibitem{Da19}
Alexandre~I. Danilenko.
\newblock Rank-one actions, their {$(C,F)$}-models and constructions with
  bounded parameters.
\newblock {\em J. Anal. Math.}, 139(2):697--749, 2019.

\bibitem{Ju77}
Andrés Del~Junco.
\newblock A transformation with simple spectrum which is not rank one.
\newblock {\em Canadian Journal of Mathematics}, 29(3):655–663, 1977.

\bibitem{DM08}
Tomasz Downarowicz and Alejandro Maass.
\newblock Finite-rank {B}ratteli-{V}ershik diagrams are expansive.
\newblock {\em Ergodic Theory Dynam. Systems}, 28(3):739--747, 2008.

\bibitem{DuPe22}
Fabien Durand and Dominique Perrin.
\newblock {\em Dimension groups and dynamical systems---substitutions,
  {B}ratteli diagrams and {C}antor systems}, volume 196 of {\em Cambridge
  Studies in Advanced Mathematics}.
\newblock Cambridge University Press, Cambridge, 2022.

\bibitem{AbLeRu14}
El~Houcein El~Abdalaoui, Mariusz Lema\'{n}czyk, and Thierry de~la Rue.
\newblock On spectral disjointness of powers for rank-one transformations and
  {M}\"{o}bius orthogonality.
\newblock {\em J. Funct. Anal.}, 266(1):284--317, 2014.

\bibitem{Fe97}
S\'{e}bastien Ferenczi.
\newblock Systems of finite rank.
\newblock {\em Colloq. Math.}, 73(1):35--65, 1997.

\bibitem{FH}
S\'{e}bastien Ferenczi and Pascal Hubert.
\newblock Rigidity of square-tiled interval exchange transformations.
\newblock {\em J. Mod. Dyn.}, 14:153--177, 2019.

\bibitem{Fo02}
N.~Pytheas Fogg.
\newblock {\em Substitutions in dynamics, arithmetics and combinatorics},
  volume 1794 of {\em Lecture Notes in Mathematics}.
\newblock Springer-Verlag, Berlin, 2002.
\newblock Edited by V. Berth\'{e}, S. Ferenczi, C. Mauduit and A. Siegel.

\bibitem{GaHi16B}
Su~Gao and Aaron Hill.
\newblock Bounded rank-1 transformations.
\newblock {\em Journal d'Analyse Mathématique}, 129(1):341--365, 2016.

\bibitem{GaHi16AT}
Su~Gao and Aaron Hill.
\newblock Topological isomorphism for rank-1 systems.
\newblock {\em Journal d'Analyse Mathématique}, 128(1):1--49, 2016.

\bibitem{GaZi18}
Su~Gao and Caleb Ziegler.
\newblock Topological mixing properties of rank-one subshifts.
\newblock {\em Trans. London Math. Soc.}, 6(1):1--21, 2019.

\bibitem{Ka84}
Steven~Arthur Kalikow.
\newblock Twofold mixing implies threefold mixing for rank one transformations.
\newblock {\em Ergodic Theory and Dynamical Systems}, 4(2):237--259, 1984.

\bibitem{Ke68}
M.~Keane.
\newblock Generalized {M}orse sequences.
\newblock {\em Z. Wahrscheinlichkeitstheorie und Verw. Gebiete}, 10:335--353,
  1968.

\bibitem{Ki88}
Jonathan~L. King.
\newblock Joining-rank and the structure of finite rank mixing transformations.
\newblock {\em J. Analyse Math.}, 51:182--227, 1988.

\bibitem{KiTh91}
Jonathan~L. King and Jean-Paul Thouvenot.
\newblock A canonical structure theorem for finite joining-rank maps.
\newblock {\em J. Analyse Math.}, 56:211--230, 1991.

\bibitem{Kr03}
Petr Kurka.
\newblock {\em Topological and symbolic dynamics}, volume~11 of {\em Cours
  Sp\'{e}cialis\'{e}s [Specialized Courses]}.
\newblock Soci\'{e}t\'{e} Math\'{e}matique de France, Paris, 2003.

\bibitem{LiMa95}
Douglas Lind and Brian Marcus.
\newblock {\em An introduction to symbolic dynamics and coding}.
\newblock Cambridge University Press, Cambridge, 1995.

\bibitem{MH40}
Marston Morse and Gustav~A. Hedlund.
\newblock Symbolic dynamics {II}. {S}turmian trajectories.
\newblock {\em Amer. J. Math.}, 62:1--42, 1940.

\bibitem{ORW82}
Donald~S. Ornstein, Daniel~J. Rudolph, and Benjamin Weiss.
\newblock Equivalence of measure preserving transformations.
\newblock {\em Mem. Amer. Math. Soc.}, 37(262):xii+116, 1982.

\bibitem{Pa22}
Ronnie Pavlov.
\newblock On subshifts with slow forbidden word growth.
\newblock {\em Ergodic Theory Dynam. Systems}, 42(4):1487--1516, 2022.

\bibitem{Qu10}
Martine Queff\'{e}lec.
\newblock {\em Substitution dynamical systems---spectral analysis}, volume 1294
  of {\em Lecture Notes in Mathematics}.
\newblock Springer-Verlag, Berlin, second edition, 2010.

\bibitem{Th20}
J\"{o}rg~M. Thuswaldner.
\newblock {$S$}-adic sequences: a bridge between dynamics, arithmetic, and
  geometry.
\newblock In {\em Substitution and tiling dynamics: introduction to
  self-inducing structures}, volume 2273 of {\em Lecture Notes in Math.}, pages
  97--191. Springer, Cham, [2020] \copyright 2020.

\end{thebibliography}
\bibliographystyle{plain}

\bigskip\bigskip

\end{document}